\documentclass[12pt]{article}

\usepackage{amssymb, amsmath, amsthm, amscd}

\usepackage[utf8]{inputenc}

\usepackage[all,cmtip]{xy}

\usepackage{enumitem}

\usepackage{setspace}

\usepackage{mathdots}

\usepackage[mathscr]{euscript}

\usepackage[colorlinks=true]{hyperref}

\hypersetup{colorlinks   = true}

\hypersetup{linkcolor=blue}

\usepackage{tikz}

\begin{document}

\mathchardef\mhyphen="2D

\newtheorem{The}{Theorem}[section]

\newtheorem{Lem}[The]{Lemma}

\newtheorem{Prop}[The]{Proposition}

\newtheorem{Cor}[The]{Corollary}

\newtheorem{Rem}[The]{Remark}

\newtheorem{Obs}[The]{Observation}

\newtheorem{SConj}[The]{Standard Conjecture}

\newtheorem{Titre}[The]{\!\!\!\! }

\newtheorem{Conj}[The]{Conjecture}

\newtheorem{Question}[The]{Question}

\newtheorem{Prob}[The]{Problem}

\newtheorem{Def}[The]{Definition}

\newtheorem{Not}[The]{Notation}

\newtheorem{Claim}[The]{Claim}

\newtheorem{Conc}[The]{Conclusion}

\newtheorem{Ex}[The]{Example}

\newtheorem{Fact}[The]{Fact}

\newtheorem{Formula}[The]{Formula}

\newtheorem{Formulae}[The]{Formulae}

\newtheorem{The-Def}[The]{Theorem and Definition}

\newtheorem{Prop-Def}[The]{Proposition and Definition}

\newtheorem{Lem-Def}[The]{Lemma and Definition}

\newtheorem{Cor-Def}[The]{Corollary and Definition}

\newtheorem{Conc-Def}[The]{Conclusion and Definition}

\newtheorem{Terminology}[The]{Note on terminology}

\newcommand{\C}{\mathbb{C}}

\newcommand{\R}{\mathbb{R}}

\newcommand{\N}{\mathbb{N}}

\newcommand{\Z}{\mathbb{Z}}

\newcommand{\Q}{\mathbb{Q}}

\newcommand{\Proj}{\mathbb{P}}

\newcommand{\Rc}{\mathcal{R}}

\newcommand{\Oc}{\mathcal{O}}

\newcommand{\Vc}{\mathcal{V}}

\newcommand{\Id}{\operatorname{Id}}

\newcommand{\pr}{\operatorname{pr}}

\newcommand{\rk}{\operatorname{rk}}

\newcommand{\del}{\partial}

\newcommand{\delbar}{\bar{\partial}}

\newcommand{\Cdot}{{\raisebox{-0.7ex}[0pt][0pt]{\scalebox{2.0}{$\cdot$}}}}

\newcommand\nilm{\Gamma\backslash G}

\newcommand\frg{{\mathfrak g}}

\newcommand{\fg}{\mathfrak g}

\newcommand{\Oh}{\mathcal{O}}

\newcommand{\Kur}{\operatorname{Kur}}

\newcommand\gc{\frg_\mathbb{C}}

\newcommand\slawek[1]{{\textcolor{red}{#1}}}

\newcommand\dan[1]{{\textcolor{blue}{#1}}}

\newcommand\question[1]{{\textcolor{green}{#1}}}

\def\pm{P_m(X,\omega,\chi)}

\def\fmv{F_m[\chi+i\partial\bar{\partial}v]}

\def\fmp{F_m[\chi+i\partial\bar{\partial}\varphi]}

\def\pa{\partial}

\begin{center}

  {\Large\bf $m$-Positivity and Regularisation}

\end{center}

\begin{center}

{\large S\l{}awomir Dinew and Dan Popovici}

\end{center}

\vspace{1ex}

\noindent{\small{\bf Abstract.} Starting from the notion of $m$-plurisubharmonic function introduced recently by Dieu and studied, in particular, by Harvey and Lawson, we consider $m$-(semi-)positive $(1,\,1)$-currents and Hermitian holomorphic line bundles on complex Hermitian manifolds and prove two kinds of results: vanishing theorems and $L^2$-estimates for the $\bar\partial$-equation in the context of $C^\infty$ $m$-positive Hermitian fibre metrics; global and local regularisation theorems for $m$-semi-positive $(1,\,1)$-currents whose proofs involve the use of viscosity subsolutions for a certain Monge-Amp\`ere-type equation and the associated Dirichlet problem. The results of the latter type require techniques of a substantially different nature from those of the classical $m=1$ case and have already been applied in subsequent joint work of the present authors and in further work of the second-named author with K. Broder.}

\vspace{1ex}

\section{Introduction}\label{section:introd} The main object of study in this paper is described in the following definition. It fits the mould of the notion of $m$-plurisubharmonic ($m$-psh for short) function introduced by Dieu in [Die06] and further studied by Harvey and Lawson in [HL13], Verbitsky in [Ver10] and Dinew in [Din22].

\begin{Def}\label{Def:m-semi-pos} Let $(X,\,\omega)$ be a complex Hermitian manifold with $\mbox{dim}_\C X = n$. Let $m\in\{1,\dots , n\}$ and let $T$ be a real current of bidegree $(1,\,1)$ on $X$.

  \vspace{1ex}

  (i)\, We say that $T$ is {\bf $m$-semi-positive (resp. $m$-positive) with respect to $\omega$} if the bidegree-$(m,\,m)$-current $T\wedge\omega^{m-1}$ is {\bf strongly semi-positive (resp. strongly positive)} on $X$. We denote this property by $T\geq_{m,\,\omega}0$ (resp. $T>_{m,\,\omega}0$).

  \vspace{1ex}

  (ii)\, We say that $T$ is {\bf $m$-semi-negative (resp. $m$-negative) with respect to $\omega$} if $-T$ is {\bf $m$-semi-positive (resp. $m$-positive) with respect to $\omega$}. We denote this property by $T\leq_{m,\,\omega}0$ (resp. $T<_{m,\,\omega}0$).

\end{Def}

It was observed in [Pop25, Corollary 3.2] that, when $T$ is continuous, $T\wedge\omega^{m-1}$ is {\it strongly semi-positive} if and only if it is {\it weakly semi-positive}. For arbitrary currents $T$ though, we will cleave to the {\it strong} semi-positivity assumption made in the above definition.

The notion of {\bf strongly positive} bidegree-$(m,\,m)$-current $S$ on an $n$-dimensional complex manifold $X$ was introduced in Definition 3.25 of [Pop25]. When $X$ is {\it compact}, it means that the bidegree-$(m,\,m)$-current $S - \varepsilon\,\omega^m$ is {\it strongly semi-positive} (in the classical sense defined by Lelong) for some (hence any) Hermitian metric $\omega$ on $X$ and some constant $\varepsilon>0$ (depending on $\omega$). When $X$ is not compact, this property is required to hold for every compact subset $K\subset X$ with a constant $\varepsilon_K>0$ depending on $K$ and weakly semi-positive test forms with compact support contained in $K$. In the above Definition \ref{Def:m-semi-pos}, this notion of {\it strong positivity} is required of bidegree-$(m,\,m)$-currents $S$ of the special shape $S=T\wedge\omega^{m-1}$, where $T$ is a $(1,\,1)$-current. For more details on the new notions of strong and weak positivity for forms and currents, the reader is referred to the appendix of [Pop25] where they were introduced and analysed.

\vspace{2ex}

When $\varphi:U\longrightarrow\R\cup\{-\infty\}$ is an upper semi-continuous function on an open subset $U\subset X$, the notions studied in the above references can be expressed in the form of the equivalences:  \begin{eqnarray}\label{eqn::m-semi-pos_equivalence-potential}\nonumber\varphi \hspace{2ex} \mbox{is} \hspace{1ex}  m\mbox{-psh} \hspace{1ex} \mbox{on} \hspace{1ex} U \hspace{1ex} \mbox{with respect to} \hspace{1ex} \omega & \iff & T:=i\partial\bar\partial\varphi\geq_{m,\,\omega}0 \hspace{1ex} \mbox{on} \hspace{1ex} U \\
 \varphi \hspace{2ex} \mbox{is} \hspace{1ex} \mbox{strictly} \hspace{1ex}  m\mbox{-psh} \hspace{1ex} \mbox{on} \hspace{1ex} U \hspace{1ex} \mbox{with respect to} \hspace{1ex} \omega & \iff &  T:=i\partial\bar\partial\varphi>_{m,\,\omega}0 \hspace{1ex} \mbox{on} \hspace{1ex} U.\end{eqnarray}

When $m=1$, $1$-(semi-)positivity for $(1,\,1)$-currents, respectively $1$-plurisubharmonicity for upper semi-continuous functions, is equivalent to the classical notion of (semi-)positivity for $(1,\,1)$-currents, respectively plurisubharmonicity for upper semi-continuous functions. These classical notions are independent of the Hermitian metric $\omega$, which does not feature in their definitions.

However, when $m\geq 2$, these notions depend on the choice of $\omega$ -- see Remark \ref{Rem:omega-dependence}.

\vspace{2ex}

An alternative description of the notion of {\it $m$-semi-positivity} for continuous currents is given in the following observation (whose proof is given in $\S$\ref{section:m-semi-pos}) that was implicit in [HL13].

\begin{Lem}\label{Lem:m-pos_eigenvalues} If $T$ is continuous in the setting of Definition \ref{Def:m-semi-pos}, the following equivalence holds: \begin{eqnarray}\label{eqn:m-pos_eigenvalues}T\geq_{m,\,\omega}0 \iff \lambda_1 + \dots + \lambda_m\geq 0 \hspace{3ex}\mbox{at every point of}\hspace{1ex} X,\end{eqnarray} where $\lambda_1\leq\dots\leq\lambda_m\leq\dots\leq\lambda_n$ are the eigenvalues of $T$ with respect to $\omega$.

\end{Lem}

\vspace{2ex}

It seems natural to extend the $m$-semi-positivity notion to holomorphic line bundles in order to give it greater geometric significance.

\begin{Def}\label{Def:m-semi-pos_line-bundle} Let $L\longrightarrow X$ be a holomorphic line bundle over a complex manifold with $\mbox{dim}_\C X = n$. Let $m\in\{1,\dots , n\}$.

 We say that $L$ is {\bf $m$-semi-positive with a $C^\infty$ metric} (respectively, {\bf $m$-semi-positive with a singular metric}) if there exist a $C^\infty$ Hermitian metric $\omega$ on $X$ and a $C^\infty$ Hermitian fibre metric (respectively a singular Hermitian fibre metric) $h$ on $L$ such that the curvature form (respectively the curvature current) of the Chern connection of $(L,\,h)$ is {\bf $m$-semi-positive with respect to $\omega$}: \begin{eqnarray*} i\Theta_h(L)\geq_{m,\,\omega}0.\end{eqnarray*}

\end{Def}

The related notion of {\it uniformly $q$-positive} holomorphic line bundle appears in part $(2)$ of Definition 2.1. of [Yan18] (with a shift of $1$ in that $q$ compared to our $m$) where a strict inequality is imposed. (cf. also Lemma \ref{Lem:m-pos_eigenvalues} below).

\vspace{2ex}

The purpose of this work is to study two aspects of $m$-(semi-)positivity for $(1,\,1)$-currents, upper semi-continuous functions and (both $C^\infty$ and singular) Hermitian holomorphic line bundles. Further aspects are studied in the sequel [DP25] to this paper, where, among other things, further generalisations of $m$-(semi-)positivity to line bundles (i.e. generalisations of Definition \ref{Def:m-semi-pos_line-bundle}) are introduced.

\vspace{2ex}

Firstly, we study in $\S$\ref{section:m-semi-pos} the consequences of the existence of an $m$-semi-positive or an $m$-semi-negative $C^\infty$ Hermitian fibre metric $h$ on a holomorphic line bundle $L$ over a compact complex manifold $X$. Using the Bochner-Kodaira-Nakano inequality relating two Laplacians and the curvature form $i\Theta_h(L)$, we get two kinds of results when $X$ admits a K\"ahler metric $\omega$:

\vspace{1ex}

(a)\, {\bf Vanishing theorems} of the following type (see Theorem \ref{The:vanishing_m-pos_Kaehler} for a more precise statement):

\begin{The}\label{The:vanishing_m-pos_Kaehler_introd} (i)\, In bidegree $(n,\,\cdot\,)$, under the curvature $q$-positivity assumption $i\Theta_h(L)\geq_{q,\,\omega} c\,\omega$ on $X$ for some $q\in\{1,\dots , n\}$ and some constant $c>0$, we have: \begin{eqnarray}\label{eqn:vanishing_m-pos_nq_introd}H^{n,\,l}_{\bar\partial}(X,\,L) = \{0\}   \hspace{5ex} \mbox{for all} \hspace{2ex} l\geq q.\end{eqnarray}

\vspace{1ex}

(ii)\, In bidegree $(\,\cdot\,,\,0)$, under the curvature $(n-p)$-negativity assumption $i\Theta_h(L)\leq_{n-p,\,\omega} -c\,\omega$ on $X$ for some $p\in\{0,\dots , n-1\}$ and some constant $c>0$, we have: \begin{eqnarray}\label{eqn:vanishing_m-pos_p0_introd}H^{l,\,0}_{\bar\partial}(X,\,L) = \{0\}  \hspace{5ex} \mbox{for all} \hspace{2ex} l\leq p.\end{eqnarray}

\end{The}

Similar results are obtained in Theorem \ref{The:vanishing_m-pos_Hermitian} when the K\"ahler assumption on the Hermitian metric $\omega$ of $X$ is dropped. The price to pay in this case is the replacement of $L$ by sufficiently high powers $L^k$ in order to ensure the cohomology vanishing.

These results are classical (the Kodaira vanishing theorem in bidegree $(n,\,q)$ and the more general Akizuki-Nakano vanishing theorem in bidegree $(p,\,q)$ with $p+q\geq n+1$ or $p+q\leq n-1$ -- see e.g. [Dem97, VII - 3.3.]) in the case of $1$-positive or $1$-negative line bundles (i.e. when $i\Theta_h(L)\geq c\,\omega$ or $i\Theta_h(L)\leq -c\,\omega$). Meanwhile, the classical Andreotti-Grauert vanishing theorem (see e.g. [Dem97, VII - 5.1.]) holds under a partial positivity assumption on the line bundle involved. Theorems \ref{The:vanishing_m-pos_Kaehler_introd} and \ref{The:vanishing_m-pos_Kaehler} translate these classical results into the $m$-positivity language used in this paper. It turns out that $m$-positivity is exactly the condition under which these statements hold.

\vspace{1ex}

(b)\, {\bf Resolution and $L^2$-estimates on the solutions of the $\bar\partial$-equation for $L$-valued forms} of the following type (see Theorem \ref{The:L2-estimates_m-pos_compact-Kaehler} for a more precise statement):

\begin{The}\label{The:L2-estimates_m-pos_compact-Kaehler_introd} (i)\, In bidegree $(n,\,\cdot\,)$, under the curvature $q$-positivity assumption $i\Theta_h(L)\geq_{q,\,\omega} c\,\omega$ for some $q\in\{1,\dots , n\}$ and some constant $c>0$, we have: for every $l\geq q$ and every $v\in C^\infty_{n,\,l}(X,\,L)$ such that $\bar\partial v=0$, there exists $u\in C^\infty_{n,\,l-1}(X,\,L)$ such that $\bar\partial u = v$ and \begin{eqnarray*}\label{eqn:L2-estimate_compact-K-nq}\int\limits_X|u|^2_{\omega,\,h}dV_\omega\leq\frac{1}{c\,l}\,\int\limits_X|v|^2_{\omega,\,h}dV_\omega.\end{eqnarray*}

\vspace{1ex}

(ii)\, In bidegree $(0,\,\cdot\,)$, under the curvature $(n-q)$-negativity assumption $i\Theta_h(L)\leq_{n-q,\,\omega} -c\,\omega$ for some $q\in\{1,\dots , n-1\}$ and some constant $c>0$, we have: for every $l\leq q$ and every $v\in C^\infty_{0,\,l}(X,\,L)$ such that $\bar\partial v=0$, there exists $u\in C^\infty_{0,\,l-1}(X,\,L)$ such that $\bar\partial u = v$ and \begin{eqnarray*}\label{eqn:L2-estimate_compact-K-0q}\int\limits_X|u|^2_{\omega,\,h}dV_\omega\leq\frac{1}{c(n-l)}\,\int\limits_X|v|^2_{\omega,\,h}dV_\omega.\end{eqnarray*}

\end{The}

Similar results are obtained in Theorem \ref{The:L2-estimates_m-pos_complete-Kaehler} when the complex manifold $X$ is not necessarily compact, but the K\"ahler metric $\omega$ thereon is {\it complete}. The right-hand-side term $v$ and the solution $u$ of the equation $\bar\partial u = v$ are, in general, only $L^2$ in this case and the proof is different from the one given in the compact case.

As with the above vanishing theorems under (a), these results are classical (H\"ormander's $L^2$ estimates and Demailly's extension of them to complete K\"ahler manifolds  -- see e.g. [Dem97, VIII - 4.5.]) in the case of $1$-positive line bundles. Generalisations to the case of partially positive line bundles exist too (see e.g. [Dem97, VIII - 6.5.]). Our Theorems \ref{The:L2-estimates_m-pos_compact-Kaehler_introd} and \ref{The:L2-estimates_m-pos_compact-Kaehler} further generalise and translate these classical results into the $m$-positivity language used in this paper.

\vspace{2ex}

Secondly, we give in $\S$\ref{section:regularisation} regularisation theorems for $m$-semi-positive $(1,1)$-currents $T$ on complex Hermitian manifolds $(X,\,\omega)$. These are key in reducing many problems involving, for example, {\it singular} Hermitian holomorphic line bundles and their curvature currents, to their counterparts (such as those dealt with in $\S$\ref{section:m-semi-pos}) in the context of $C^\infty$ Hermitian fibre metrics -- see [BP25] for an application of such a scheme.

Our results extend to $m$-semi-positive $(1,1)$-currents and $m$-psh upper semi-continuous functions classical regularisation results known in the case $m=1$. The proofs in the case where $m\geq 2$ treated in this paper are very different from those of their classical counterparts for the case $m=1$.

\vspace{1ex}

In $\S$\ref{subsection:global-regularisation}, we solve the global regularisation problem on a compact complex Hermitian manifold $(X,\,\omega)$ under the assumption that the Bott-Chern cohomology class of the $d$-closed $m$-semi-positive $(1,\,1)$-current $T$, that we wish to realise as the weak limit of a sequence of $T$-cohomologous $C^\infty$ $m$-semi-positive $(1,\,1)$-forms, contains an $m$-positive $C^\infty$ representative $\chi$. When $m=1$, this hypothesis means that the class of $T$ contains a K\"ahler metric. Thus, our result generalises to the case $m\geq 2$ the result of B\l{}ocki and Ko\l{}odziej in [BK07] stipulating that a semi-positive $(1,\,1)$-current $T$, whose Bott-Chern cohomology class is K\"ahler, can be approximated by cohomologous $C^\infty$ $(1,\,1)$-forms without the loss of positivity (depending on the Lelong numbers of $T$) inevitable in the case of general cohomology classes covered by Demailly's current-regularisation theorem of [Dem92].

Our global regularisation result is the following (see Theorem \ref{The:global-approximation} for a more precise statement):

\begin{The}\label{The:global-approximation_introd} Under the circumstances described above, for any upper semi-continuous $L^1_{loc}$ function $\varphi:X\longrightarrow\R\cup\{-\infty\}$ such that $\chi + i\partial\bar\partial\varphi\geq_{m,\,\omega}0$ on $X$, there exists a sequence $(\varphi_j)_{j\geq 1}$ of $C^\infty$ functions $\varphi_j:X\longrightarrow\R$ such that:

\vspace{1ex}

(i)\, $\chi + i\partial\bar\partial\varphi_j\geq_{m,\,\omega}0$ on $X$ for every $j\geq 1$;

\vspace{1ex}

(ii)\, for every $x\in X$, the sequence $(\varphi_j(x))_{j\geq 1}$ of reals is non-increasing and converges to $\varphi(x)$ as $j\to\infty$.

\end{The}

\vspace{1ex}

In $\S$\ref{subsection:local-regularisation}, we solve the local regularisation problem on an open subset $U_p\subset X$ (neighbourhood of an arbitrary given point $p\in X$) of a complex manifold $X$. In the classical $m=1$ case, this problem has a well-known solution: every psh function $\varphi\not\equiv -\infty$ on a connected open subset $\Omega\Subset\C^n$ is the pointwise limit $\varphi\star\rho_\varepsilon\searrow\varphi$ (as $\varepsilon\searrow 0$) on $\Omega_\varepsilon:=\{z\in\Omega\,\mid\,d(z,\,\partial\Omega)>\varepsilon\}\Subset\Omega$ of the $C^\infty$ psh functions obtained by convoluting $\varphi$ with a family $(\rho_\varepsilon)_\varepsilon$ of smoothing kernels. However, in the case where $m\geq 2$, the $C^\infty$ functions $\varphi\star\rho_\varepsilon$ need not be $m$-psh w.r.t. a given Hermitian metric $\omega$ when $\varphi$ is, except in the very special case where the coefficients of $\omega$ are all constant. The constancy of the coefficients of $\omega$ cannot be ensured in local coordinates on a complex manifold since it is not invariant under coordinate changes. So, a completely different method had to be introduced to handle the case $m\geq 2$.

Our local regularisation result is the following (cf. Theorem \ref{The:local-approximation}):

\begin{The}\label{The:local-approximation_introd} Let $(X,\,\omega)$ be a complex Hermitian manifold. Then, for every point $p\in X$, every open neighbourhood $U_p\subset X$ of $p$ and every $m$-psh (with respect to $\omega$) function $u:U_p\longrightarrow\R\cup\{-\infty\}$, there exists a sequence $(u_j)_{j\geq 1}$ of $C^\infty$ strictly $m$-psh (with respect to $\omega$) functions $u_j:U_p\longrightarrow\R$ such that the sequence $(u_j(x))_{j\geq 1}$ of reals is non-increasing and converges to $u(x)$ for every $x\in U_p$.

\end{The}

\vspace{2ex}

In the proofs of both our global and local regularisation theorems, we use the theory of viscosity subsolutions to a certain Monge-Amp\`ere-type equation that was solved by Cheng and Xu in [ChX25]. In both cases, we need to consider a Dirichlet problem associated with that equation. The resolution of this Dirichlet problem is quite technically involved in the case of local regularisations (cf. Theorem \ref{Thm:Dirichlet-Fm}). In spite of a host of existing references for similar problems, none seems to state and solve the precise Dirichlet problem we need. Especially in the local case, there seems to be no general machinery for providing the crucial {\it normal-normal} estimates on the boundary -- see [GN18], [CP22], [GGQ22] and [Do23] for approaches using equation-specific methods in similar Dirichlet problems. Therefore, we provided a complete proof in the Appendix ($\S$\ref{section:Appendix}).

\vspace{1ex}

{\bf Acknowledgments}. The first-named author would like to thank N. C. Nguyen for illuminating discussions on [GN18]. He was partially supported by grant no. 2021/41/B/ST1/01632 from the National Science Center, Poland.

\section{$m$-semi-positivity of $(1,\,1)$-currents and line bundles}\label{section:m-semi-pos}

Definition \ref{Def:m-semi-pos} is generalised in the expected way to the situation where two real $(1,\,1)$-currents $T,S$ are involved or where $m$-(semi-)negativity is meant: \begin{eqnarray*}T\geq_{m,\,\omega}S \hspace{3ex} (\mbox{resp.} \hspace{2ex} T>_{m,\,\omega}S) & \iff & T-S\geq_{m,\,\omega}0 \hspace{3ex} (\mbox{resp.} \hspace{2ex} T-S>_{m,\,\omega}0) \\
  T\leq_{m,\,\omega}0 \hspace{3ex} (\mbox{resp.} \hspace{2ex} T<_{m,\,\omega}0) & \iff &  -T\geq_{m,\,\omega}0 \hspace{3ex} (\mbox{resp.} \hspace{2ex} -T>_{m,\,\omega}0).\end{eqnarray*}

\vspace{1ex}

\begin{Rem}\label{Rem:omega-dependence} The notion of {\bf $m$-semi-positivity} with respect to a Hermitian metric $\omega$ {\bf depends} on the choice of $\omega$ when $m\geq 2$. To see this, consider the following {\bf example}. In $\C^3$, take \begin{eqnarray*}\omega = a\,idz_1\wedge d\bar{z}_1 + b\,idz_2\wedge d\bar{z}_2 + c\,idz_3\wedge d\bar{z}_3 \hspace{3ex}\mbox{and}\hspace{3ex} T = p\,idz_1\wedge d\bar{z}_1 + q\,idz_2\wedge d\bar{z}_2 + r\,idz_3\wedge d\bar{z}_3,\end{eqnarray*} with $a,b,c>0$ and $p,q,r\in\R$. Then, $T\wedge\omega$ equals \begin{eqnarray*}(pb+qa)\,idz_1\wedge d\bar{z}_1\wedge idz_2\wedge d\bar{z}_2 + (pc+ra)\,idz_1\wedge d\bar{z}_1\wedge idz_3\wedge d\bar{z}_3 + (qc+rb)\,idz_2\wedge d\bar{z}_2\wedge idz_3\wedge d\bar{z}_3,\end{eqnarray*} so we get the equivalence: \begin{eqnarray*}T\geq_{2,\,\omega}0 \iff pb+qa\geq 0, \hspace{2ex} pc+ra\geq 0  \hspace{2ex}\mbox{and}\hspace{2ex} qc+rb\geq 0.\end{eqnarray*} In particular, the current $T$ defined by the coefficients $p=q=1$ and $r=-1$ satisfies the condition $T\geq_{2,\,\omega}0$ when, for example, $a=b=1$ and $c=3$, but this $T$ is not $2$-semi-positive with respect to $\omega$ when $c<a$ or $c<b$.

\end{Rem}

\vspace{2ex}

As for Lemma \ref{Lem:m-pos_eigenvalues}, stated in the introduction, recall that, whenever $T$ is of class $C^2$, for every point $x\in X$, there exist local holomorphic coordinates $z_1,\dots , z_n$ on $X$ centred at $x$ such that \begin{eqnarray}\label{eqn:coordinates_diagonalisation}\omega(x) = \sum\limits_{j=1}^n idz_j\wedge d\bar{z}_j  \hspace{5ex}\mbox{and}\hspace{5ex}  T(x) = \sum\limits_{j=1}^n \lambda_j(x)\,idz_j\wedge d\bar{z}_j.\end{eqnarray} The functions $\lambda_j:X\longrightarrow\R$ are independent of the choices of local holomorphic coordinates $z_1,\dots , z_n$ simultaneously diagonalising $\omega$ and $T$ as above. They are the eigenvalues of $T$ with respect to $\omega$. After possibly permuting the coordinates, we may order them non-decreasingly. Lemma \ref{Lem:m-pos_eigenvalues} ensures that the $m$-semi-positivity condition on $T$ is equivalent to the sum of any $m$ of its eigenvalues being non-negative at every point of $X$.

\vspace{2ex}

\noindent {\it Proof of Lemma \ref{Lem:m-pos_eigenvalues}.} We fix an arbitrary point $X\in X$ and choose local holomorphic coordinates $z_1,\dots , z_n$ centred at $x$ with property (\ref{eqn:coordinates_diagonalisation}). We get: \begin{eqnarray*}\omega_{m-1}(x) = \frac{\omega^{m-1}(x)}{(m-1)!} = \sum\limits_{1\leq k_1<\dots<k_{m-1}\leq n} idz_{k_1}\wedge d\bar{z}_{k_1}\wedge\dots\wedge idz_{k_{m-1}}\wedge d\bar{z}_{k_{m-1}}\end{eqnarray*}

\begin{eqnarray*} = \sum\limits_{|K|=m-1}i^{(m-1)^2}dz_K\wedge d\bar{z}_K,\end{eqnarray*} where $K=(1\leq k_1<\dots<k_{m-1}\leq n)$ are multi-indices of length $m-1$ and $dz_K:=dz_{k_1}\wedge\dots\wedge dz_{k_{m-1}}$.

Hence: \begin{eqnarray*}T\wedge\omega_{m-1}(x) =  \sum\limits_{j=1}^n \sum\limits_{\stackrel{|K|=m-1}{j\notin K}}i^{m^2}\lambda_j(x)\,dz_{jK}\wedge d\bar{z}_{jK} = \sum\limits_{|J|=m}\bigg(\sum\limits_{j\in J}\lambda_j\bigg)\,i^{m^2} dz_J\wedge d\bar{z}_J,\end{eqnarray*} where $jK$ is the multi-index of length $m$ obtained by adding $j$ to $K$ and re-ordering the indices increasingly if need be. Since each $i^{m^2} dz_J\wedge d\bar{z}_J = idz_{j_1}\wedge d\bar{z}_{j_1}\wedge\dots\wedge idz_{j_m}\wedge d\bar{z}_{j_m}$ is a decomposable $(m,\,m)$-form, we conclude that the strong semi-positivity of $T\wedge\omega_{m-1}(x)$ is equivalent to the non-negativity of each of its coefficients: $\sum_{j\in J}\lambda_j\geq 0$ for every $J$ with $|J|=m$.  \hfill $\Box$

\vspace{2ex}

Another immediate observation is that $m$-(semi-)positivity and $m$-(semi-)negativity can only become weaker properties as $m$ increases.

\begin{Lem}\label{Lem:weaker_m-increases} In the setting of Definition \ref{Def:m-semi-pos}, the following implications hold: \begin{eqnarray*}T\geq_{m,\,\omega}0 \hspace{3ex} (\mbox{resp.} \hspace{2ex} T>_{m,\,\omega}0) & \implies & T\geq_{m+1,\,\omega}0 \hspace{3ex} (\mbox{resp.} \hspace{2ex} T>_{m+1,\,\omega}0) \\
 T\leq_{m,\,\omega}0 \hspace{3ex} (\mbox{resp.} \hspace{2ex} T<_{m,\,\omega}0) & \implies & T\leq_{m+1,\,\omega}0 \hspace{3ex} (\mbox{resp.} \hspace{2ex} T<_{m+1,\,\omega}0).\end{eqnarray*}

\end{Lem}

\begin{proof} If the $(m,\,m)$-current $T\wedge\omega^{m-1}$ is strongly semi-positive (resp. strongly positive), it remains so after multiplication by any strongly positive form, in particular by $\omega$.

 The analogous statement holds when the word ``positive'' is replaced by the word ``negative''.

\end{proof}

Let us recall the following standard fact: every function $\varphi$ that is $m$-psh with respect to a Hermitian metric $\omega$ is also {\it subharmonic} with respect to $\omega$ in the sense that $\Delta_\omega\varphi\geq 0$, where $\Delta_\omega\varphi:=\Lambda_\omega(i\partial\bar\partial\varphi)$. This follows at once from Lemma \ref{Lem:weaker_m-increases} and from the well-known equality $(\Lambda_\omega\alpha)\,\omega_n = \alpha\wedge\omega_{n-1}$ that holds for every $(1,\,1)$-form $\alpha$, where we set $\alpha_p:=\alpha^p/p!$ for any positive integer $p$.

\subsection{A consequence of the BKN identity}\label{subsection:BKN-consequence} We start by recalling some standard material that will be needed. Let $(E,\,h)\longrightarrow (X,\,\omega)$ be a Hermitian holomorphic vector bundle of rank $r\geq 1$ over a complex Hermitian manifold of dimension $n$. Denoting by $D_h=D_h' + \bar\partial$ the Chern connection of $(E,\,h)$, by $\Lambda_\omega := (\omega\wedge\cdot\,)^\star$ the adjoint w.r.t. the pointwise inner product $\langle\,\cdot\,,\,\cdot\,\rangle_\omega$ induced by $\omega$ on the differential forms on $X$ and by $\tau=\tau_\omega = [\Lambda_\omega,\,\partial\omega\wedge\cdot\,]$ the torsion operator induced by $\omega$, one has the following

\vspace{2ex}

\noindent {\bf Bochner-Kodaira-Nakano-type identity (Demailly [Dem84])} \begin{eqnarray}\label{eqn:BKN_identity}\Delta'' = \Delta'_\tau + [i\Theta_h(E)\wedge\cdot\,,\,\Lambda_\omega] + T_{\omega},\end{eqnarray} where \begin{eqnarray*}\Delta'_\tau &:= & (D'_h + \tau_\omega)\,(D'_h + \tau_\omega)^\star_{\omega,\,h} + (D'_h + \tau_\omega)^\star_{\omega,\,h}\,(D'_h + \tau_\omega): C^\infty_{p,\,q}(X,\,L)\longrightarrow C^\infty_{p,\,q}(X,\,L), \\
  \Delta'' & := & \bar\partial\bar\partial^\star_{\omega,\,h} + \bar\partial^\star_{\omega,\,h}\bar\partial : C^\infty_{p,\,q}(X,\,L)\longrightarrow C^\infty_{p,\,q}(X,\,L), \\
  T_{\omega}  & :=  & \bigg[\Lambda_\omega,\,[\Lambda_\omega,\,\frac{i}{2}\,\partial\bar\partial\omega]\bigg] - [\partial\omega\wedge\cdot,\,(\partial\omega\wedge\cdot)^{\star}]: C^\infty_{p,\,q}(X,\,L)\longrightarrow C^\infty_{p,\,q}(X,\,L), \\
  i\Theta_h(E)\wedge\cdot & : & C^\infty_{p,\,q}(X,\,L)\longrightarrow C^\infty_{p+1,\,q+1}(X,\,L),\end{eqnarray*} are, respectively, the Laplace-type operators induced by the type-$(1,\,0)$-operator $D'_h + \tau_\omega$ and the type-$(0,\,1)$-operator $\bar\partial$; a zero-th order operator of type $(0,\,0)$ associated with the torsion of $\omega$; the operator of multiplication by the curvature form $i\Theta_h(E)\in C^\infty(X,\,\mbox{End}\,(E))$ of $D_h$, all acting on $L$-valued forms of any bidegree $(p,\,q)$ on $X$.

\vspace{2ex}

Expressions of the type $[i\Theta_h(E)\wedge\cdot\,,\,\Lambda_\omega]$ featuring on the right of (\ref{eqn:BKN_identity}) can be handled in an effective way by means of the following well-known formula:

\begin{Lem}([Dem97, VI. $\S5.2.$])\label{Lem:curvature-operator_coordinates} For any Hermitian metric $\omega$ and any real $(1,\,1)$-form $T$, for every point $x\in X$ about which local holomorphic coordinates $z_1,\dots , z_n$ have been chosen to simultaneously diagonalise $\omega$ and $T$ as in (\ref{eqn:coordinates_diagonalisation}) and for every form $u=\sum\limits_{J,\,K}u_{J\overline{K}}\,dz_J\wedge d\bar{z}_K$, one has: \begin{eqnarray}\label{eqn:curvature-operator_coordinates}[T\wedge\cdot\,,\,\Lambda_\omega]\,u = \sum\limits_{J,\,K}\bigg(\sum\limits_{j\in J}\lambda_j + \sum\limits_{k\in K}\lambda_k - \sum\limits_{l=1}^n\lambda_l\bigg)\,u_{J\overline{K}}\,dz_J\wedge d\bar{z}_K    \hspace{7ex}\mbox{at}\hspace{1ex} x.\end{eqnarray}

\end{Lem}

\vspace{2ex}

We now observe the following consequence of this fact.

\begin{Lem}\label{Lem:l-bounds_curv-op} Let $(L,\,h)\longrightarrow (X,\,\omega)$ be a holomorphic line bundle over a compact complex Hermitian manifold with $\mbox{dim}_\C X = n$, where $h$ is a $C^\infty$ Hermitian fibre metric. For any $(p,\,q)$, let \begin{eqnarray*}A_{\omega,\,h}:=[i\Theta_h(L)\wedge\cdot\,,\,\Lambda_\omega]:\Lambda^{p,\,q}T^\star X\otimes L\longrightarrow\Lambda^{p,\,q}T^\star X\otimes L\end{eqnarray*} be the curvature operator associated with $(L,\,h)$ and acting in bidegree $(p,\,q)$.

\vspace{1ex}

$(1)$\, For any $p\in\{0,\dots , n\}$, if there exists a constant $c>0$ such that $i\Theta_h(L)\leq_{n-p,\,\omega} -c\,\omega$ everywhere on $X$, then $A_{\omega,\,h}\geq c(n-l)\,\mbox{Id}_{\Lambda^{l,\,0}T^\star X\otimes L}$ in bidegree $(l,\,0)$, namely: \begin{eqnarray}\label{eqn:l-bound_curv-op_p0}\bigg\langle A_{\omega,\,h} u,\,u\bigg\rangle_{\omega,\,h}\geq c(n-l)\,|u|^2_{\omega,\,h},   \hspace{5ex} \mbox{for all}\hspace{2ex} u\in\Lambda^{l,\,0}T^\star_x X\otimes L_x, \hspace{1ex} x\in X,\end{eqnarray} for every $l\leq p$.

\vspace{1ex}

$(2)$\, For any $q\in\{0,\dots , n\}$, if there exists a constant $c>0$ such that $i\Theta_h(L)\leq_{n-q,\,\omega} -c\,\omega$ everywhere on $X$, then $A_{\omega,\,h}\geq c(n-l)\,\mbox{Id}_{\Lambda^{0,\,l}T^\star X\otimes L}$ in bidegree $(0,\,l)$, namely: \begin{eqnarray}\label{eqn:l-bound_curv-op_0q}\bigg\langle A_{\omega,\,h} u,\,u\bigg\rangle_{\omega,\,h}\geq c(n-l)\,|u|^2_{\omega,\,h},   \hspace{5ex} \mbox{for all}\hspace{2ex} u\in\Lambda^{0,\,l}T^\star_x X\otimes L_x, \hspace{1ex} x\in X,\end{eqnarray} for every $l\leq q$.

\vspace{1ex}

$(3)$\, For any $q\in\{0,\dots , n\}$, if there exists a constant $c>0$ such that $i\Theta_h(L)\geq_{q,\,\omega} c\,\omega$ everywhere on $X$, then $A_{\omega,\,h}\geq c\,l\,\mbox{Id}_{\Lambda^{n,\,l}T^\star X\otimes L}$ in bidegree $(n,\,l)$, namely: \begin{eqnarray}\label{eqn:l-bound_curv-op_nq}\bigg\langle A_{\omega,\,h} u,\,u\bigg\rangle_{\omega,\,h}\geq c\,l\,|u|^2_{\omega,\,h},   \hspace{5ex} \mbox{for all}\hspace{2ex} u\in\Lambda^{n,\,l}T^\star_x X\otimes L_x, \hspace{1ex} x\in X,\end{eqnarray} for every $l\geq q$.

\vspace{1ex}

$(4)$\, For any $p\in\{0,\dots , n\}$, if there exists a constant $c>0$ such that $i\Theta_h(L)\geq_{p,\,\omega} c\,\omega$ everywhere on $X$, then $A_{\omega,\,h}\geq c\,l\,\mbox{Id}_{\Lambda^{l,\,n}T^\star X\otimes L}$ in bidegree $(l,\,n)$, namely: \begin{eqnarray}\label{eqn:l-bound_curv-op_pn}\bigg\langle A_{\omega,\,h} u,\,u\bigg\rangle_{\omega,\,h}\geq c\,l\,|u|^2_{\omega,\,h},   \hspace{5ex} \mbox{for all}\hspace{2ex} u\in\Lambda^{l,\,n}T^\star_x X\otimes L_x, \hspace{1ex} x\in X,\end{eqnarray} for every $l\geq p$.

\end{Lem}

\begin{proof} Fix an arbitrary point $x\in X$ and choose local holomorphic coordinates $z_1,\dots , z_n$ on $X$ centred at $x$ such that (\ref{eqn:coordinates_diagonalisation}) holds for $\omega$ and $T:=i\Theta_h(L)$. Let $\{e\}$ be a local frame for $L$ near $x$.

\vspace{1ex}

$(1)$\, When $q=0$, let $u\in \Lambda^{p,\,0}T^\star_x X\otimes L_x$ be arbitrary and let $u=\sum\limits_{|J|=p}u_J\,dz_J\otimes e$ be its expression in the chosen local coordinates. Thus, formula (\ref{eqn:curvature-operator_coordinates}) applied with $T:=i\Theta_h(L)$ yields at $x$: \begin{eqnarray}\label{eqn:curv-op_pointwise-l-bound_p0}\nonumber\bigg\langle A_{\omega,\,h} u,\,u\bigg\rangle_{\omega,\,h} & = & \sum\limits_{|J|=p}\bigg(\sum\limits_{j\in J}\lambda_j - \sum\limits_{l=1}^n\lambda_l\bigg)\,|u_J|^2\,|e|^2_h = -\sum\limits_{|J|=p}\bigg(\sum\limits_{j\in C_J}\lambda_j\bigg)\,|u_J|^2\,|e|^2_h \\
  & \geq & c(n-p)\,|u|_{\omega,\,h}^2,\end{eqnarray} where the inequality followed from the curvature hypothesis $i\Theta_h(L)\leq_{n-p,\,\omega} -c\,\omega$. Indeed, (\ref{eqn:coordinates_diagonalisation}) implies: \begin{eqnarray*}\bigg(i\Theta_h(L) + c\,\omega\bigg)(x) = \sum\limits_{j=1}^n \bigg(\lambda_j(x) + c\bigg)\,idz_j\wedge d\bar{z}_j,\end{eqnarray*} hence Lemma \ref{Lem:m-pos_eigenvalues} ensures that our curvature hypothesis is equivalent at $x$ to the following inequality: \begin{eqnarray*}\sum\limits_{j\in C_J}\lambda_j(x) + c\,(n-p)\leq 0 \hspace{5ex} \mbox{for all}\hspace{1ex} J \hspace{1ex} \mbox{such that}\hspace{1ex} |J|=p.\end{eqnarray*}

Since $x\in X$ is arbitrary, inequality (\ref{eqn:curv-op_pointwise-l-bound_p0}), which is the same as (\ref{eqn:l-bound_curv-op_p0}) for $l=p$, holds everywhere on $X$. When $l<p$, we have $n-l>n-p$, so Lemma \ref{Lem:weaker_m-increases} ensures that the hypothesis $i\Theta_h(L)\leq_{n-p,\,\omega} -c\,\omega$ implies $i\Theta_h(L)\leq_{n-l,\,\omega} -c\,\omega$. Thus, (\ref{eqn:curv-op_pointwise-l-bound_p0}) applied with $l$ in place of $p$ proves the claim.

\vspace{1ex}

$(2)$\, Similarly, when $p=0$ and $u=\sum\limits_{|K|=q}u_K\,d\bar{z}_K\otimes e\in \Lambda^{0,\,q}T^\star_x X\otimes L_x$, formula (\ref{eqn:curvature-operator_coordinates}) yields at $x$: \begin{eqnarray*}\bigg\langle A_{\omega,\,h} u,\,u\bigg\rangle_{\omega,\,h} & = & \sum\limits_{|K|=q}\bigg(\sum\limits_{k\in K}\lambda_k - \sum\limits_{l=1}^n\lambda_l\bigg)\,|u_K|^2\,|e|^2_h = -\sum\limits_{|K|=q}\bigg(\sum\limits_{k\in C_K}\lambda_k\bigg)\,|u_K|^2\,|e|^2_h \\
  & \geq & c(n-q)\,|u|_{\omega,\,h}^2,\end{eqnarray*} where the inequality followed from the curvature hypothesis $i\Theta_h(L)\leq_{n-q,\,\omega} -c\,\omega$. This proves (\ref{eqn:l-bound_curv-op_0q}) for $l=q$. When $l<q$, $n-l>n-q$, so Lemma \ref{Lem:weaker_m-increases} enables one to apply (\ref{eqn:l-bound_curv-op_0q}) with $l$ in place of $q$.

\vspace{1ex}

$(3)$\, When $p=n$ and $u=\sum\limits_{|K|=q}u_K\,dz_1\wedge\dots\wedge dz_n\wedge d\bar{z}_K\otimes e\in \Lambda^{n,\,q}T^\star_x X\otimes L_x$, (\ref{eqn:curvature-operator_coordinates}) yields at $x$: \begin{eqnarray*}\bigg\langle A_{\omega,\,h} u,\,u\bigg\rangle_{\omega,\,h} & = & \sum\limits_{|K|=q}\bigg(\sum\limits_{k\in K}\lambda_k\bigg)\,|u_K|^2\,|e|^2_h \geq cq\,|u|_{\omega,\,h}^2,\end{eqnarray*} where the inequality followed from the curvature hypothesis $i\Theta_h(L)\geq_{q,\,\omega} c\,\omega$. This proves (\ref{eqn:l-bound_curv-op_nq}) for $l=q$. When $l>q$, Lemma \ref{Lem:weaker_m-increases} ensures that the hypothesis $i\Theta_h(L)\geq_{q,\,\omega} c\,\omega$ implies $i\Theta_h(L)\geq_{l,\,\omega} c\,\omega$, so (\ref{eqn:l-bound_curv-op_nq}) applies with $l$ in place of $q$.

\vspace{1ex}

$(4)$\, When $q=n$ and $u=\sum\limits_{|J|=p}u_J\,dz_J\wedge d\bar{z}_1\wedge\dots\wedge d\bar{z}_n\otimes e\in \Lambda^{p,\,n}T^\star_x X\otimes L_x$, (\ref{eqn:curvature-operator_coordinates}) yields at $x$: \begin{eqnarray*}\bigg\langle A_{\omega,\,h} u,\,u\bigg\rangle_{\omega,\,h} & = & \sum\limits_{|J|=p}\bigg(\sum\limits_{j\in J}\lambda_j\bigg)\,|u_J|^2\,|e|^2_h \geq cp\,|u|_{\omega,\,h}^2,\end{eqnarray*} where the inequality followed from the curvature hypothesis $i\Theta_h(L)\geq_{p,\,\omega} c\,\omega$. This proves (\ref{eqn:l-bound_curv-op_pn}) for $l=p$. When $l>p$, Lemma \ref{Lem:weaker_m-increases} ensures that the hypothesis $i\Theta_h(L)\geq_{p,\,\omega} c\,\omega$ implies $i\Theta_h(L)\geq_{l,\,\omega} c\,\omega$, so (\ref{eqn:l-bound_curv-op_pn}) applies with $l$ in place of $p$.

\end{proof}

\subsection{Vanishing theorems for $m$-positive and $m$-negative line bundles}\label{subsection:vanishing_m-pos} The following result is our vanishing theorem in the K\"ahler setting.

\begin{The}\label{The:vanishing_m-pos_Kaehler} Let $L\longrightarrow X$ be a holomorphic line bundle over a compact complex manifold with $\mbox{dim}_\C X = n$.

  If there exist a {\bf K\"ahler} metric $\omega$ on $X$ and a $C^\infty$ Hermitian fibre metric $h$ on $L$ such that:

  \vspace{1ex}

  $(1)$\, $i\Theta_h(L)\leq_{n-p,\,\omega} -c\,\omega$ on $X$ for some $p\in\{0,\dots , n-1\}$ and some constant $c>0$, then \begin{eqnarray}\label{eqn:vanishing_m-pos_p0}H^{l,\,0}_{\bar\partial}(X,\,L) = \{0\}  \hspace{5ex} \mbox{for all} \hspace{2ex} l\leq p.\end{eqnarray}

  Under the weaker curvature assumption $i\Theta_h(L)\leq_{n-p,\,\omega} 0$ on $X$, for every $l\leq p$, every $\Gamma\in H^{l,\,0}_{\bar\partial}(X,\,L)$ has the following properties with respect to the Chern connection $D_h=D'_h + \bar\partial$ of $(L,\,h)$: \begin{eqnarray}\label{eqn:parallel_m-pos_p0}D'_h\Gamma = 0 \hspace{5ex}\mbox{and}\hspace{5ex} (D'_h)^\star\Gamma = 0.\end{eqnarray} In particular, every such $\Gamma$ is $D_h$-parallel: $D_h\Gamma = 0$.

  \vspace{1ex}

  $(2)$\, $i\Theta_h(L)\leq_{n-q,\,\omega} -c\,\omega$ on $X$ for some $q\in\{0,\dots , n-1\}$ and some constant $c>0$, then \begin{eqnarray}\label{eqn:vanishing_m-pos_0q}H^{0,\,l}_{\bar\partial}(X,\,L) = \{0\} \hspace{5ex} \mbox{for all} \hspace{2ex} l\leq q.\end{eqnarray}

  Under the weaker curvature assumption $i\Theta_h(L)\leq_{n-q,\,\omega} 0$ on $X$, for every $l\leq q$, every $\Gamma\in H^{0,\,l}_{\bar\partial}(X,\,L)$ has the property: $D'_h\Gamma = 0$ (hence also $D_h\Gamma = 0$).

  \vspace{1ex}

  $(3)$\, $i\Theta_h(L)\geq_{q,\,\omega} c\,\omega$ on $X$ for some $q\in\{1,\dots , n\}$ and some constant $c>0$, then \begin{eqnarray}\label{eqn:vanishing_m-pos_nq}H^{n,\,l}_{\bar\partial}(X,\,L) = \{0\}   \hspace{5ex} \mbox{for all} \hspace{2ex} l\geq q.\end{eqnarray}

  Under the weaker curvature assumption $i\Theta_h(L)\geq_{q,\,\omega} 0$ on $X$, for every $l\geq q$, every $\Gamma\in H^{n,\,l}_{\bar\partial}(X,\,L)$ has the properties: $D'_h\Gamma = 0$ (hence also $D_h\Gamma = 0$) and $(D'_h)^\star\Gamma = 0$.

  \vspace{1ex}

  $(4)$\, $i\Theta_h(L)\geq_{p,\,\omega} c\,\omega$ on $X$ for some $p\in\{1,\dots , n\}$ and some constant $c>0$, then \begin{eqnarray}\label{eqn:vanishing_m-pos_pn}H^{l,\,n}_{\bar\partial}(X,\,L) = \{0\} \hspace{5ex} \mbox{for all} \hspace{2ex} l\geq p.\end{eqnarray}

  Under the weaker curvature assumption $i\Theta_h(L)\geq_{p,\,\omega} 0$ on $X$, for every $l\geq p$, every $\Gamma\in H^{l,\,n}_{\bar\partial}(X,\,L)$ has the properties: $D'_h\Gamma = 0$ (hence also $D_h\Gamma = 0$) and $(D'_h)^\star\Gamma = 0$.

\end{The}

\noindent {\it Proof.} Since $\omega$ is K\"ahler, the torsion operators $\tau_\omega$ and $T_\omega$ vanish identically, so the Bochner-Kodaira-Nakano formula (\ref{eqn:BKN_identity}) reduces to \begin{eqnarray}\label{eqn:BKN_identity_Kaehler}\Delta'' = \Delta' + [i\Theta_h(L)\wedge\cdot\,,\,\Lambda_\omega]\end{eqnarray} as operators acting on $L$-valued $(p,\,q)$-forms on $X$ for every bidegree $(p,\,q)$, where $\Delta':=  D'_h\,(D'_h)^\star_{\omega,\,h} + (D'_h)^\star_{\omega,\,h}\,D'_h$. Therefore, for every form $u\in C^\infty_{p,\,q}(X,\,L)$, after discarding the non-negative quantity $\langle\langle\Delta'u,\,u\rangle\rangle_{\omega,\,h}\geq 0$, we get the following well-known Bochner-Kodaira-Nakano inequality: \begin{eqnarray}\label{eqn:BKN_inequality_Kaehler}\langle\langle\Delta''u,\,u\rangle\rangle_{\omega,\,h}\geq\int\limits_X\bigg\langle[i\Theta_h(L)\wedge\cdot\,,\,\Lambda_\omega]\,u,\,u\bigg\rangle_{\omega,\,h}\,dV_\omega.\end{eqnarray}

On the other hand, since $\bar\partial^2=0$, $\Delta''$ is elliptic and $X$ is compact, the standard harmonic theory yields the Hodge isomorphism in every bidegree $(p,\,q)$: \begin{eqnarray}\label{eqn:Hodge_isom}H^{p,\,q}_{\bar\partial}(X,\,L)\simeq{\cal H}^{p,\,q}_{\Delta''}(X,\,L):=\{u\in C^\infty_{p,\,q}(X,\,L)\,\mid\,\Delta''u = 0\}.\end{eqnarray} Thus, proving the vanishing of the cohomology group $H^{p,\,q}_{\bar\partial}(X,\,L)$ amounts to proving the vanishing of the $\Delta''$-harmonic space ${\cal H}^{p,\,q}_{\Delta''}(X,\,L)$.

\vspace{1ex}

$(1)$\, When $q=0$, it suffices to prove the statements in the case where $l=p$ since Lemma \ref{Lem:weaker_m-increases} can then be used in conjunction with the case $l=p$ to yield the cases where $l<p$.

Let $\Gamma\in{\cal H}^{p,\,0}_{\Delta''}(X,\,L)$ be arbitrary. We know from $(1)$ of Lemma \ref{Lem:l-bounds_curv-op} that, under our curvature assumption, we have: \begin{eqnarray*}\bigg\langle[i\Theta_h(L)\wedge\cdot\,,\,\Lambda_\omega]\,\Gamma,\,\Gamma\bigg\rangle_{\omega,\,h}\geq c(n-p)\,|\Gamma|^2_{\omega,\,h},   \hspace{5ex} \mbox{everywhere on}\hspace{2ex} X.\end{eqnarray*}




 Integrating on $X$ with respect to $dV_\omega$ and using (\ref{eqn:BKN_inequality_Kaehler}), we get: \begin{eqnarray*}0 = \langle\langle\Delta''\Gamma,\,\Gamma\rangle\rangle\geq\int\limits_X\bigg\langle[i\Theta_h(L)\wedge\cdot\,,\,\Lambda_\omega]\,\Gamma,\,\Gamma\bigg\rangle_{\omega,\,h}\,dV_\omega\geq c\,(n-p)\,||\Gamma||_{\omega,\,h}^2\geq 0,\end{eqnarray*} where the equality followed from $\Delta''\Gamma = 0$. This implies $\Gamma = 0$. Hence, ${\cal H}^{p,\,0}_{\Delta''}(X,\,L) = \{0\}$.

If we only assume that $i\Theta_h(L)\leq_{n-p,\,\omega} 0$ everywhere on $X$ (i.e. $c=0$), we get successively: \begin{eqnarray*}\bigg\langle[i\Theta_h(L)\wedge\cdot\,,\,\Lambda_\omega]\,\Gamma,\,\Gamma\bigg\rangle_{\omega,\,h} & \geq & 0, \\
  0 = \langle\langle\Delta''\Gamma,\,\Gamma\rangle\rangle_{\omega,\,h} & \geq & \int\limits_X\bigg\langle[i\Theta_h(L)\wedge\cdot\,,\,\Lambda_\omega]\,\Gamma,\,\Gamma\bigg\rangle_{\omega,\,h}\,dV_\omega\geq  0,\end{eqnarray*} hence both inequalities on the last line above must be equalities. In view of (\ref{eqn:BKN_identity_Kaehler}) and (\ref{eqn:BKN_inequality_Kaehler}), this implies that $\langle\langle\Delta'\Gamma,\,\Gamma\rangle\rangle = 0$, which amounts to $D'_h\Gamma = 0$ and $(D'_h)^\star_{\omega,\,h}\Gamma = 0$ thanks to $X$ being compact. This proves (\ref{eqn:parallel_m-pos_p0}).

\vspace{2ex}

Parts $(2)$, $(3)$ and $(4)$ can be proved similarly.  \hfill $\Box$

\vspace{3ex}

\vspace{2ex}

We now give a variant of Theorem \ref{The:vanishing_m-pos_Kaehler} in the general context of a {\bf possibly non-K\"ahler} manifold $X$. The price to pay is that the vanishing results will hold for various cohomology groups with values in sufficiently high tensor powers $L^k$ (also denoted $kL$) of $L$, rather than in $L$ itself.

\begin{The}\label{The:vanishing_m-pos_Hermitian} Let $L\longrightarrow X$ be a holomorphic line bundle over a compact complex manifold with $\mbox{dim}_\C X = n$.

  If there exist a {\bf Hermitian} metric $\omega$ on $X$ and a $C^\infty$ Hermitian fibre metric $h$ on $L$ such that:

  \vspace{1ex}

  $(1)$\, $i\Theta_h(L)\leq_{n-p,\,\omega} -c\,\omega$ on $X$ for some $p\in\{0,\dots , n-1\}$ and some constant $c>0$, then \begin{eqnarray}\label{eqn:vanishing_m-pos_p0_nK}H^{l,\,0}_{\bar\partial}(X,\,L^k) = \{0\} \hspace{5ex} \mbox{for all} \hspace{2ex} l\leq p \hspace{2ex}\mbox{and all} \hspace{2ex} k\gg 1.\end{eqnarray}

\vspace{1ex}

$(2)$\, $i\Theta_h(L)\leq_{n-q,\,\omega} -c\,\omega$ on $X$ for some $q\in\{0,\dots , n-1\}$ and some constant $c>0$, then \begin{eqnarray}\label{eqn:vanishing_m-pos_0q_nK}H^{0,\,l}_{\bar\partial}(X,\,L^k) = \{0\} \hspace{5ex} \mbox{for all} \hspace{2ex} l\leq q \hspace{2ex}\mbox{and all} \hspace{2ex} k\gg 1.\end{eqnarray}

\vspace{1ex}

$(3)$\, $i\Theta_h(L)\geq_{q,\,\omega} c\,\omega$ on $X$ for some $q\in\{1,\dots , n\}$ and some constant $c>0$, then \begin{eqnarray}\label{eqn:vanishing_m-pos_nq_nK}H^{n,\,l}_{\bar\partial}(X,\,L^k) = \{0\} \hspace{5ex} \mbox{for all} \hspace{2ex} l\geq q \hspace{2ex}\mbox{and all} \hspace{2ex} k\gg 1.\end{eqnarray}

\vspace{1ex}

$(4)$\, $i\Theta_h(L)\geq_{p,\,\omega} c\,\omega$ on $X$ for some $p\in\{1,\dots , n\}$ and some constant $c>0$, then \begin{eqnarray}\label{eqn:vanishing_m-pos_pn_nK}H^{l,\,n}_{\bar\partial}(X,\,L^k) = \{0\} \hspace{5ex} \mbox{for all} \hspace{2ex} l\geq p \hspace{2ex}\mbox{and all} \hspace{2ex} k\gg 1.\end{eqnarray}

\end{The}

\noindent {\it Proof.} It is similar to the proof of the K\"ahler counterpart (Theorem \ref{The:vanishing_m-pos_Kaehler}) of this result, so we will only point out the new arguments.

The key point is that the torsion operator $T_\omega$ in Demailly's Bochner-Kodaira-Nakano-type identity (\ref{eqn:BKN_identity}) is of order zero (hence bounded) and depends only on the Hermitian metric $\omega$ (hence remains unchanged when we replace the line bundle $E=L$ by $E=L^k$). Therefore, it will be made irrelevant by the curvature term $[i\Theta_h(L^k)\wedge\cdot\,,\,\Lambda_\omega]$ when $k$ is large enough.

Explicitly, since $i\Theta_h(L^k) = k\,i\Theta_h(L)$, identity (\ref{eqn:BKN_identity}), when applied to $E=L^k$, reads: \begin{eqnarray}\label{eqn:BKN_identity_Hermitian_k}\Delta_k'' = \Delta'_{\tau,\,k} + k\,[i\Theta_h(L)\wedge\cdot\,,\,\Lambda_\omega] + T_{\omega},   \hspace{5ex} k\geq 1,\end{eqnarray} where $\Delta_k''$ and $\Delta'_{\tau,\,k}$ are the counterparts of $\Delta''$ and $\Delta'_\tau$ when $(L,\,h)$ and its Chern connection are replaced by $(L^k,\,h^k)$ and its Chern connection.

Since the operator $T_\omega$ is bounded, there exists a constant $C_\omega>0$, depending only on $\omega$, such that, at every point $x\in X$, we have: \begin{eqnarray*}|\langle T_\omega u,\,u\rangle_\omega|\leq C_\omega\,|u|_\omega^2, \hspace{5ex} p,q\in\{0,\dots , n\}, \hspace{1ex} u\in\Lambda^{p,\,q}T^\star_x X.\end{eqnarray*}

Therefore, for every bidegree $(p,\,q)$ and every $\Delta_k''$-harmonic $L^k$-valued $(p,\,q)$-form $u\in{\cal H}^{p,\,q}_{\Delta_k''}(X,\,L^k)\simeq H^{p,\,q}_{\bar\partial}(X,\,L^k)$, the Bochner-Kodaira-Nakano-type identity (\ref{eqn:BKN_identity_Hermitian_k}) yields (cf. its K\"ahler counterpart (\ref{eqn:BKN_inequality_Kaehler})): \begin{eqnarray*}0& = &\langle\langle\Delta_k''u,\,u\rangle\rangle_{\omega,\,h^k} \\ & = & \langle\langle\Delta'_{\tau,\,k}u,\,u\rangle\rangle_{\omega,\,h^k} + k\,\int\limits_X\bigg\langle[i\Theta_h(L)\wedge\cdot\,,\,\Lambda_\omega]\,u,\,u\bigg\rangle_{\omega,\,h^k}\,dV_\omega + \int\limits_X\bigg\langle T_\omega\,u,\,u\bigg\rangle_{\omega,\,h^k}\,dV_\omega     \\
  & \geq & k\,\int\limits_X\bigg\langle[i\Theta_h(L)\wedge\cdot\,,\,\Lambda_\omega]\,u,\,u\bigg\rangle_{\omega,\,h^k}\,dV_\omega - C_\omega\,||u||_{\omega,\,h^k}^2.\end{eqnarray*}

\vspace{1ex}

Thus, in bidegree $(p,\,0)$, for every $u = \Gamma\in{\cal H}^{p,\,0}_{\Delta_k''}(X,\,L^k)\simeq H^{p,\,0}_{\bar\partial}(X,\,L^k)$, we get: \begin{eqnarray*}0=\langle\langle\Delta_k''\Gamma,\,\Gamma\rangle\rangle_{\omega,\,h^k} \geq \bigg(c(n-p)\,k - C_\omega  \bigg)\,||\Gamma||_{\omega,\,h^k}^2\end{eqnarray*} after using $(1)$ of Lemma \ref{Lem:l-bounds_curv-op}. Since $c(n-p)>0$, $c(n-p)\,k - C_\omega>0$ for all $k$ sufficiently large. For those $k$, we conclude that $\Gamma = 0$. This proves (\ref{eqn:vanishing_m-pos_p0_nK}) in the case $l=p$. The cases $l<p$ are deduced from this one in conjunction with Lemma \ref{Lem:weaker_m-increases}.

  Identities (\ref{eqn:vanishing_m-pos_0q_nK})--(\ref{eqn:vanishing_m-pos_pn_nK}) can be proved similarly.  \hfill $\Box$

\subsection{$L^2$-estimates for the $\bar\partial$-equation on $m$-positive and $m$-negative line bundles}\label{subsection:L2-estimates}

We shall now use the results and techniques of $\S$\ref{subsection:vanishing_m-pos} to give H\"ormander-type $L^2$-estimates in various situations. The first case is the one of $C^\infty$ forms and compact K\"ahler manifolds.

\begin{The}\label{The:L2-estimates_m-pos_compact-Kaehler} Let $L\longrightarrow X$ be a holomorphic line bundle over a {\bf compact} complex manifold with $\mbox{dim}_\C X = n$.

\vspace{1ex}

  (a)\, Suppose there exist a {\bf K\"ahler} metric $\omega$ on $X$ and a $C^\infty$ Hermitian fibre metric $h$ on $L$ such that $i\Theta_h(L)\leq_{n-q,\,\omega} -c\,\omega$ on $X$ for some $q\in\{1,\dots , n-1\}$ and some constant $c>0$.

  Then, for every $l\leq q$ and every $v\in C^\infty_{0,\,l}(X,\,L)$ such that $\bar\partial v=0$, there exists $u\in C^\infty_{0,\,l-1}(X,\,L)$ such that $\bar\partial u = v$ and \begin{eqnarray*}\label{eqn:L2-estimate_compact-K-0q}\int\limits_X|u|^2_{\omega,\,h}dV_\omega\leq\frac{1}{c(n-l)}\,\int\limits_X|v|^2_{\omega,\,h}dV_\omega.\end{eqnarray*}

\vspace{1ex}

  (b)\, Suppose there exist a {\bf K\"ahler} metric $\omega$ on $X$ and a $C^\infty$ Hermitian fibre metric $h$ on $L$ such that $i\Theta_h(L)\geq_{q,\,\omega} c\,\omega$ on $X$ for some $q\in\{1,\dots , n\}$ and some constant $c>0$.

Then, for every $l\geq q$ and every $v\in C^\infty_{n,\,l}(X,\,L)$ such that $\bar\partial v=0$, there exists $u\in C^\infty_{n,\,l-1}(X,\,L)$ such that $\bar\partial u = v$ and \begin{eqnarray*}\label{eqn:L2-estimate_compact-K-nq}\int\limits_X|u|^2_{\omega,\,h}dV_\omega\leq\frac{1}{c\,l}\,\int\limits_X|v|^2_{\omega,\,h}dV_\omega.\end{eqnarray*}

\vspace{1ex}

(c)\, Suppose there exist a {\bf K\"ahler} metric $\omega$ on $X$ and a $C^\infty$ Hermitian fibre metric $h$ on $L$ such that $i\Theta_h(L)\geq_{p,\,\omega} c\,\omega$ on $X$ for some $p\in\{1,\dots , n\}$ and some constant $c>0$.

Then, for every $l\geq p$ and every $v\in C^\infty_{l,\,n}(X,\,L)$ such that $\bar\partial v=0$, there exists $u\in C^\infty_{l,\,n-1}(X,\,L)$ such that $\bar\partial u = v$ and \begin{eqnarray*}\label{eqn:L2-estimate_compact-complex-K-pn}\int\limits_X|u|^2_{\omega,\,h}dV_\omega\leq\frac{1}{c\,l}\,\int\limits_X|v|^2_{\omega,\,h}dV_\omega.\end{eqnarray*}

\end{The}

\noindent {\it Proof.} For any bidegree $(p,\,q)$, any Hermitian metric $\omega$ on $X$ and any $C^\infty$ Hermitian fibre metric $h$ on $L$, the integrability property $\bar\partial^2=0$, the ellipticity of the $\bar\partial$-Laplacian $\Delta''=\Delta''_{\omega,\,h}$ and the compactness of $X$ yield, by standard harmonic theory, the $3$-space $L^2_{\omega,\,h}$-orthogonal decomposition: \begin{eqnarray*}C^\infty_{p,\,q}(X,\,L) = {\cal H}^{p,\,q}_{\Delta''}(X,\,L)\oplus\mbox{Im}\,\bar\partial\oplus\mbox{Im}\,\bar\partial^\star, \hspace{5ex}\mbox{with}\hspace{2ex} \ker\bar\partial = {\cal H}^{p,\,q}_{\Delta''}(X,\,L)\oplus\mbox{Im}\,\bar\partial.\end{eqnarray*}

In particular, $H^{p,\,q}_{\bar\partial}(X,\,L)\simeq{\cal H}^{p,\,q}_{\Delta''}(X,\,L)$. Thus, if $H^{p,\,q}_{\bar\partial}(X,\,L) = \{0\}$, then $C^\infty_{p,\,q}(X,\,L)\cap\ker\bar\partial = C^\infty_{p,\,q}(X,\,L)\cap\mbox{Im}\,\bar\partial$, so for every $v\in C^\infty_{p,\,q}(X,\,L)\cap\ker\bar\partial$, there exists $u\in C^\infty_{p,\,q-1}(X,\,L)$ such that $v = \bar\partial u$. The minimal $L^2_{\omega,\,h}$-norm solution is given by the Neumann formula: \begin{eqnarray*}u = \bar\partial^\star\Delta^{''-1} v,\end{eqnarray*} where $\Delta^{''-1}$ is the Green operator of $\Delta''$.

From this, we get: \begin{eqnarray}\label{eqn:Neumann_L2-estimate}||u||^2 = \langle\langle\bar\partial^\star\Delta^{''-1} v,\,\bar\partial^\star\Delta^{''-1} v \rangle\rangle = \langle\langle\Delta^{''-1} v,\,\bar\partial\bar\partial^\star\Delta^{''-1} v \rangle\rangle = \langle\langle\Delta^{''-1} v,\,\Delta''\Delta^{''-1} v \rangle\rangle = \langle\langle\Delta^{''-1} v,\, v \rangle\rangle,\end{eqnarray} where the third equality follows from $\bar\partial^\star\bar\partial\Delta^{''-1} v = \bar\partial^\star\Delta^{''-1}(\bar\partial v) = 0$ since $\bar\partial v = 0$. Thus, it remains to find an upper bound for $\langle\langle\Delta^{''-1} v,\, v \rangle\rangle$ in terms of $||v||^2$ in each of the cases (a), (b), (c) described in the statement.

So far, we haven't used either the K\"ahler assumption on $X$, or the positivity/negativity assumption on the curvature of $(L,\,h)$.

\vspace{1ex}

(a)\, Thanks to Lemma \ref{Lem:weaker_m-increases}, it suffices to prove the case $l=q$. When $p=0$, $q\in\{1,\dots , n\}$, $X$ is K\"ahler and $i\Theta_h(L)\leq_{n-q,\,\omega} -c\,\omega$ on $X$ for some constant $c>0$, we get, as in the proof of Theorem \ref{The:vanishing_m-pos_Kaehler}, using $(2)$ of Remark \ref{Lem:l-bounds_curv-op}: \begin{eqnarray*}\langle\langle\Delta''v,\,v\rangle\rangle\geq\int\limits_X\bigg\langle[i\Theta_h(L)\wedge\cdot\,,\,\Lambda_\omega]\,v,\,v\bigg\rangle\,dV_\omega\geq c\,(n-q)\,||v||^2\geq 0,\end{eqnarray*} for every $v\in C^\infty_{0,\,q}(X,\,L)$. This implies the inequality below: \begin{eqnarray*}||u||^2 = \langle\langle\Delta^{''-1} v,\, v \rangle\rangle\leq\frac{1}{c\,(n-q)}\,||v||^2\end{eqnarray*} for every $v\in C^\infty_{0,\,q}(X,\,L)$, in particular for our original $v\in C^\infty_{0,\,q}(X,\,L)\cap\ker\bar\partial$, where the equality featuring the  minimal $L^2_{\omega,\,h}$-norm solution $u$ of the equation $\bar\partial u = v$ is given by (\ref{eqn:Neumann_L2-estimate}). Note that this equation is solvable (i.e. $v\in\mbox{Im}\,\bar\partial$) thanks to $H^{0,\,q}_{\bar\partial}(X,\,L) = \{0\}$ (a fact that is ensured by $(2)$ of Theorem \ref{The:vanishing_m-pos_Kaehler}) and to the arguments given above.

\vspace{1ex}

Parts (b) and (c) can be proved in a similar way using the above arguments and parts $(3)$, respectively $(4)$, of Theorem \ref{The:vanishing_m-pos_Kaehler}. \hfill $\Box$

\vspace{2ex}

The second case is the one of $L^2$-forms and more general complete K\"ahler manifolds. For example, every weakly pseudo-convex K\"ahler manifold carries a complete K\"ahler metric.

\begin{The}\label{The:L2-estimates_m-pos_complete-Kaehler}  Let $(L,\,h)\longrightarrow (X,\,\omega)$ be a holomorphic line bundle over a {\bf complete K\"ahler} manifold with $\mbox{dim}_\C X = n$ and $h$ a $C^\infty$ Hermitian fibre metric.

\vspace{1ex}

(a)\, Suppose there exist a constant $c>0$ and $q\in\{1,\dots , n\}$ such that $i\Theta_h(L)\leq_{n-q,\,\omega} -c\,\omega$ on $X$. Then, for every $l\leq q$ and every $v\in L^2_{0,\,l}(X,\,L)$ such that $\bar\partial v=0$ and $||v||^2_{\omega,\,h}<\infty$, there exists $u\in L^2_{0,\,l-1}(X,\,L)$ such that $\bar\partial u = v$ and \begin{eqnarray*}\label{eqn:L2-estimate_complete-K-0q}||u||^2_{\omega,\,h} := \int\limits_X|u|^2_{\omega,\,h}dV_\omega\leq\frac{1}{c(n-l)}\,\int\limits_X|v|^2_{\omega,\,h}dV_\omega: = \frac{1}{c(n-l)}\,||v||^2_{\omega,\,h}.\end{eqnarray*}

\vspace{1ex}

(b)\, Suppose there exist a constant $c>0$ and $q\in\{1,\dots , n\}$ such that $i\Theta_h(L)\geq_{q,\,\omega} c\,\omega$ on $X$. Then, for every $l\geq q$ and every $v\in L^2_{n,\,l}(X,\,L)$ such that $\bar\partial v=0$ and $||v||^2_{\omega,\,h}<\infty$, there exists $u\in L^2_{n,\,l-1}(X,\,L)$ such that $\bar\partial u = v$ and \begin{eqnarray*}\label{eqn:L2-estimate_complete-K-nq}||u||^2_{\omega,\,h} :=\int\limits_X|u|^2_{\omega,\,h}dV_\omega\leq\frac{1}{c\,l}\,\int\limits_X|v|^2_{\omega,\,h}dV_\omega: = \frac{1}{c\,l}\,||v||^2_{\omega,\,h}.\end{eqnarray*}

\vspace{1ex}

(c)\, Suppose there exist a constant $c>0$ and $p\in\{0,\dots , n\}$ such that $i\Theta_h(L)\geq_{p,\,\omega} c\,\omega$ on $X$. Then, for every $l\geq p$ and every $v\in L^2_{l,\,n}(X,\,L)$ such that $\bar\partial v=0$ and $||v||^2_{\omega,\,h}<\infty$, there exists $u\in L^2_{l,\,n-1}(X,\,L)$ such that $\bar\partial u = v$ and \begin{eqnarray*}\label{eqn:L2-estimate_compact-complete-K-pn}||u||^2_{\omega,\,h} :=\int\limits_X|u|^2_{\omega,\,h}dV_\omega\leq\frac{1}{c\,l}\,\int\limits_X|v|^2_{\omega,\,h}dV_\omega: = \frac{1}{c\,l}\,||v||^2_{\omega,\,h}.\end{eqnarray*}

\end{The}

\noindent {\it Proof.} For any bidegree $(p,\,q)$, any {\it complete} Hermitian metric $\omega$ on $X$ and any $C^\infty$ Hermitian fibre metric $h$ on $L$, the $\C$-vector space ${\cal D}^{p,\,q}(X,\,L)$ of $C^\infty$ compactly supported $L$-valued $(p,\,q)$-forms on $X$ is {\it dense} with respect to the graph norm $u\mapsto||u||_{\omega,\,h} + ||\bar\partial u||_{\omega,\,h} + ||\bar\partial^\star_{\omega,\,h} u||_{\omega,\,h}$ in the intersection of the domains $\mbox{Dom}\,\bar\partial\cap\mbox{Dom}\,\bar\partial^\star_{\omega,\,h}\subset L^2_{p,\,q}(X,\,L)$. This means that for every $u\in\mbox{Dom}\,\bar\partial\cap\mbox{Dom}\,\bar\partial^\star_{\omega,\,h}\subset L^2_{p,\,q}(X,\,L)$ (i.e. for every $u\in L^2_{p,\,q}(X,\,L)$ such that $\bar\partial u\in L^2_{p,\,q+1}(X,\,L)$ and $\bar\partial^\star_{\omega,\,h} u\in L^2_{p,\,q-1}(X,\,L)$ when these forms are computed by deriving the coefficients of $u$ in the sense of distributions), there exists a sequence $(u_\nu)_\nu\subset{\cal D}^{p,\,q}(X,\,L)$ such that: \begin{eqnarray*}u_\nu\underset{\nu\to\infty}{\overset{L^2}\longrightarrow}u \hspace{5ex}\mbox{and}\hspace{5ex} \bar\partial u_\nu\underset{\nu\to\infty}{\overset{L^2}\longrightarrow}\bar\partial u  \hspace{3ex}\mbox{and}\hspace{3ex} \bar\partial^\star_{\omega,\,h} u_\nu\underset{\nu\to\infty}{\overset{L^2}\longrightarrow}\bar\partial^\star_{\omega,\,h} u.\end{eqnarray*} The only reason why $\omega$ (which need not be K\"ahler at this point) has to be assumed complete is to guarantee this density property which, in turn, ensures that the Bochner-Kodaira-Nakano inequality (\ref{eqn:BKN_inequality_Kaehler}) (valid when $\omega$ is K\"ahler) extends to all the forms $u\in\mbox{Dom}\,\bar\partial\cap\mbox{Dom}\,\bar\partial^\star_{\omega,\,h}\subset L^2_{p,\,q}(X,\,L)$: \begin{eqnarray}\label{eqn:BKN_ineq_proof_d-bar}||\bar\partial u||^2_{\omega,\,h} + ||\bar\partial^\star_{\omega,\,h} u||^2_{\omega,\,h}\geq\int\limits_X\bigg\langle[i\Theta_h(L)\wedge\cdot\,,\,\Lambda_\omega]\,u,\,u\bigg\rangle_{\omega,\,h}\,dV_\omega = \int\limits_X\bigg\langle A_{\omega,\,h}\,u,\,u\bigg\rangle_{\omega,\,h}\,dV_\omega,\end{eqnarray} where $A_{\omega,\,h}$ denotes the curvature operator as in Remark \ref{Lem:l-bounds_curv-op}.

Meanwhile, we have the Hilbert space $L^2_{\omega,\,h}$-orthogonal decomposition: \begin{eqnarray}\label{eqn:2-space-decomp_perp}L^2_{p,\,q}(X,\,L) = \ker\bar\partial\oplus\bigg(\ker\bar\partial\bigg)^\perp  \hspace{5ex}\mbox{with}\hspace{5ex} \bigg(\ker\bar\partial\bigg)^\perp\subset\ker\bar\partial^\star_{\omega,\,h} \hspace{2ex}\mbox{and}\hspace{2ex} \ker\bar\partial \hspace{2ex}\mbox{closed}.\end{eqnarray}

\vspace{1ex}

(a)\, Thanks to Lemma \ref{Lem:weaker_m-increases}, it suffices to prove the case $l=q$. When $p=0$, $q\in\{1,\dots , n\}$, $X$ is K\"ahler and $i\Theta_h(L)\leq_{n-q,\,\omega} -c\,\omega$ on $X$ for some constant $c>0$, for every $w\in{\cal D}^{0,\,q}(X,\,L)$, the decomposition (\ref{eqn:2-space-decomp_perp}) yields a unique splitting: \begin{eqnarray}w = w_1 + w_2  \hspace{5ex}\mbox{with}\hspace{5ex} w_1\in\ker\bar\partial \hspace{2ex}\mbox{and}\hspace{2ex} w_2\in(\ker\bar\partial)^\perp\subset\ker\bar\partial^\star_{\omega,\,h}.\end{eqnarray}

Now, fix an arbitrary form $v\in L^2_{0,\,q}(X,\,L)$ such that $\bar\partial v=0$ and $||v||^2_{\omega,\,h}<\infty$. For every $w\in{\cal D}^{0,\,q}(X,\,L)$, we get: \begin{eqnarray*}\bigg|\langle\langle v,\,w\rangle\rangle_{\omega,\,h}\bigg|^2 = \bigg|\langle\langle v,\,w_1\rangle\rangle_{\omega,\,h}\bigg|^2 & \stackrel{(i)}{\leq} & \langle\langle A_{\omega,\,h}^{-1}v,\,v\rangle\rangle_{\omega,\,h}\,\langle\langle A_{\omega,\,h}w_1,\,w_1\rangle\rangle_{\omega,\,h} \\
  & \stackrel{(ii)}{\leq} & \langle\langle A_{\omega,\,h}^{-1}v,\,v\rangle\rangle_{\omega,\,h}\,\bigg(||\bar\partial w_1||^2_{\omega,\,h} + ||\bar\partial^\star_{\omega,\,h} w_1||^2_{\omega,\,h}\bigg),\end{eqnarray*}  where:

\vspace{1ex}

$\bullet$ inequality (i) followed from the Cauchy-Schwarz inequality applied to the curvature operator that satisfies the lower estimate $A_{\omega,\,h}\geq c(n-q)\,\mbox{Id}_{\Lambda^{0,\,q}T^\star X\otimes L} > 0$ in bidegree $(0,\,q)$ everywhere on $\Lambda^{0,\,q}T^\star X\otimes L$ thanks to our curvature hypothesis and to $(2)$ of Remark \ref{Lem:l-bounds_curv-op}; in particular, we have: \begin{eqnarray}\label{eqn:A-1_upper-bound}\langle\langle A_{\omega,\,h}^{-1}v,\,v\rangle\rangle_{\omega,\,h} \leq \frac{1}{c(n-q)}\,||v||^2_{\omega,\,h}<\infty;\end{eqnarray}

\vspace{1ex}

$\bullet$ inequality (ii) followed from the Bochner-Kodaira-Nakano inequality (\ref{eqn:BKN_ineq_proof_d-bar}) applied to $w_1\in\mbox{Dom}\,\bar\partial\cap\mbox{Dom}\,\bar\partial^\star_{\omega,\,h}\subset L^2_{0,\,q}(X,\,L)$.

\vspace{1ex}

Since $\bar\partial w_1 = 0$ and $\bar\partial^\star_{\omega,\,h} w_1 = \bar\partial^\star_{\omega,\,h} w$ (because $\bar\partial^\star_{\omega,\,h}w_2 = 0$), we get: \begin{eqnarray*}\bigg|\langle\langle w,\,v\rangle\rangle_{\omega,\,h}\bigg|^2 \leq \langle\langle A_{\omega,\,h}^{-1}v,\,v\rangle\rangle_{\omega,\,h}\,||\bar\partial^\star_{\omega,\,h} w||^2_{\omega,\,h}  \hspace{5ex}\mbox{for all}\hspace{2ex} w\in{\cal D}^{0,\,q}(X,\,L).\end{eqnarray*}

This means that the linear form:

\begin{eqnarray*}L^2_{0,\,q-1}(X,\,L) \supset\bar\partial^\star_{\omega,\,h}\bigg({\cal D}^{0,\,q}(X,\,L)\bigg)\ni\bar\partial^\star_{\omega,\,h} w\longmapsto \langle\langle w,\,v\rangle\rangle_{\omega,\,h}\in\C
\end{eqnarray*} is continuous in the $L^2_{\omega,\,h}$-norm topology and its norm is $\leq(\langle\langle A_{\omega,\,h}^{-1}v,\,v\rangle\rangle_{\omega,\,h})^{1/2}<\infty$.

  The Hahn-Banach extension theorem ensures the existence of a form $u\in L^2_{0,\,q-1}(X,\,L)$ such that \begin{eqnarray*}||u||_{\omega,\,h} \leq\bigg(\langle\langle A_{\omega,\,h}^{-1}v,\,v\rangle\rangle_{\omega,\,h}\bigg)^{1/2}\leq\frac{1}{\sqrt{c(n-q)}}\,||v||_{\omega,\,h}<\infty
  \end{eqnarray*} (see (\ref{eqn:A-1_upper-bound}) for the second inequality) and \begin{eqnarray*}\langle\langle w,\,v\rangle\rangle_{\omega,\,h} = \langle\langle\bar\partial^\star_{\omega,\,h} w,\,u\rangle\rangle_{\omega,\,h}  \hspace{5ex}\mbox{for all}\hspace{2ex} w\in{\cal D}^{0,\,q}(X,\,L).\end{eqnarray*} By the density of ${\cal D}^{0,\,q}(X,\,L)$ in $\mbox{Dom}\,\bar\partial\cap\mbox{Dom}\,\bar\partial^\star_{\omega,\,h}\subset L^2_{0,\,q}(X,\,L)$ with respect to the graph norm, this last piece of information is equivalent to $\bar\partial u = v$ in the sense of distributions. This $u$ is the sought-after solution of equation $\bar\partial u = v$ under (a).

\vspace{2ex}

Parts (b) and (c) can be proved in a similar way.  \hfill $\Box$

\section{Regularisation of $m$-semi-positive $(1,1)$-currents}\label{section:regularisation}

In subsequent applications, it will be useful to approximate $m$-semi-positive $(1,1)$-currents by smooth ones while preserving $m$-semi-positivity. Given a complex Hermitian $n$-dimensional manifold $(X,\omega)$ equipped with a $C^\infty$ $d$-closed real $(1,1)$-form $\chi$, we consider currents $T$ of the form $\chi+ i\partial\bar{\partial}\varphi$, for an upper semi-continuous locally integrable function $\varphi:X\to\R\cup\{-\infty\}$, such that $T\geq_{m,\,\omega}0$ on $X$ for some $m\in\{1,\dots , n\}$. We denote the space of all such functions $\varphi$ by $P_m(X,\omega,\chi)$. The problem then naturally reduces to an approximation of $\varphi$ by $C^\infty$ functions $\varphi_j\in P_m(X,\omega,\chi)$. As it is customary for such classes of functions, we seek a {{\it non-increasing approximation}} $\varphi_j\searrow\varphi$ as $j\to\infty$. We shall deal with both the global approximation (under suitable conditions) and its local counterpart, the latter being unobstructed, as will be seen.


\subsection{Global regularisation}\label{subsection:global-regularisation}

In this subsection we assume that $X$ is additionally compact. As $T=\chi+ i\partial\bar{\partial}\varphi$ and $\chi$ share the same Bott-Chern cohomology class, the condition $T\geq_{m,\,\omega}0$ imposes restrictions on the class $[\chi]_{BC}$. In the case $m=1$, this forces the class $[\chi]_{BC}$ to be {\bf pseudo-effective}, i.e. to contain a closed semi-positive current. However, even the stronger nef assumption on the cohomology class $[\chi]_{BC}$, let alone its mere pseudo-effectivity, does not suffice to guarantee approximation of semi-positive $(1,\,1)$-currents $T=\chi+ i\partial\bar{\partial}\varphi$ by semi-positive $C^\infty$ $(1,\,1)$-forms within their Bott-Chern cohomology class -- see [DPS94, Example 1.7] for an illuminating example. Therefore, we make the following {\bf assumption}: \begin{equation}\label{alpha-m-positive}
	\chi>_{m,\,\omega}0.
	\end{equation}

Note that in the classical case (i.e. $m=1$), this is exactly the K\"ahlerianity condition.

\vspace{1ex}

When $X$ is $\mathbb C^n$ (or a domain therein), $\chi=0$  and $\omega$ is the {\bf standard} K\"ahler metric $\omega=\sum_{k=1}^n idz_k\wedge d\bar{z}_k$, it is a classical fact that such a regularisation can be locally obtained (in the case $m=1$) through convolutions with a family of smoothing kernels. Unfortunately, similar {techniques only carry over to the case $m\geq 2$} in very special cases -- for example if $X$ is the quotient of a simple algebraic group $G$ by a parabolic subgroup $H$ and $\omega$ is fixed by the maximal compact subgroup of $G$. The problem is that the $m$-psh property depends on the metric $\omega$ when $m\geq 2$ and the translates used in the convolution, regardless of the way they are defined, require testing for $m$-positivity at a pregiven point, whereas the only information we have is at the translated point. For convolution-based approximations done in a local chart, the shifts fill in only a small ball about the base point, hence the discrepancy between $\omega^{m-1}$ computed at the base point and at the translated point is small. This simple observation accounts for {\bf strictly} $m$-plurisubharmonic functions being easy to approximate -- see [Ver10] (where the same idea is carried out by means of the heat flow). For mere $m$-plurisubharmonic functions, this argument fails. The same difficulty has been observed for other function classes on complex manifolds, see [Pli13], [LN15], [HLP16] or [ChX25]. Thus, we will have to enforce subtler methods.

\vspace{1ex}

Our approach will be modelled on the argument of Chinh and Nguyen from [LN15] who dealt with a different type of $m$-positive functions. The key difference is that in the setting of [LN15] a nonlinear potential theory can be built which, coupled with the Berman thermodynamical formalism from [Ber19], allows one to define {\it continuous} envelopes. In our setting however, the natural nonlinear operator associated with $P_m(X,\omega,\chi)$, called $F_m$ in the definition below, does not allow a simple analogous potential theory. We circumvent this difficulty by applying the {\it viscosity} machinery developed recently by Cheng and Xu -- see [ChX25]. Actually, all we shall need is the smooth solvability of the $F_m$-equation for smooth data and a simple lemma for viscosity subsolutions -- see Lemma \ref{Lem:visc-subsolutions} below.

\begin{Def}\label{Def:F_m}[ChX25, Page 3] Let $X,\ \omega,\ \chi$ be as above. For any $C^2$ function $v\in\pm$, let $\lambda_1\leq\dots\leq\lambda_n$ be the eigenvalues of the $(1,\,1)$-form $\chi+i\partial\bar{\partial}v$ with respect to $\omega$ -- see (\ref{eqn:coordinates_diagonalisation}). Then, the function $\fmv:X\longrightarrow\R$ is defined at every point $z\in X$ by the formula:

\begin{equation*}\fmv(z):=\Big(\Pi_{|J|=m}\Big[\sum_{j\in J}\lambda_j(z)\Big]\Big)^{1/\binom nm},\end{equation*} where $J=(j_1,\cdots,j_m)$ runs over all ordered $m$-tuples of indices from $\lbrace1,\cdots,n\rbrace$.

\end{Def}

As noticed in Lemma \ref{Lem:m-pos_eigenvalues} and its proof, one has: \begin{eqnarray*}(\chi+i\partial\bar{\partial}v)\wedge\omega_{m-1}(x) = \sum\limits_{|J|=m}\bigg(\sum\limits_{j\in J}\lambda_j(x)\bigg)\,i^{m^2} dz_J\wedge d\bar{z}_J\end{eqnarray*} at every point $x\in X$ about which the local coordinates $z_j$ have been chosen such that $\omega$ and $T:=\chi+i\partial\bar{\partial}v$ are simultaneously diagonal as in (\ref{eqn:coordinates_diagonalisation}). Thus, the very definition of $\pm$ implies that, for every $C^2$ function $v\in\pm$, all the sums $\sum_{j\in J}\lambda_j$ with $|J|=m$ are non negative. Moreover, one can easily check (see [ChX25] for the details) that \begin{eqnarray*}\pm\cap C^2(X,\,\R)\ni v\longmapsto\fmv\in C^0(X,\,\R)\end{eqnarray*} is a nonlinear {\bf elliptic} operator. This enables one to introduce the {\bf viscosity machinery} for $F_m$ (compare [ChX25, Definition 2.3]):

\begin{Def}\label{Def:viscosity_subsolution}[ChX25, Definition 2.3] Let $g$ be a non-negative continuous function on $X$. A function $\varphi\in\pm$ is said to be a {\bf viscosity subsolution} of the equation \begin{eqnarray*}\fmv=g\end{eqnarray*} if for any point $p\in X$ and any $C^2$ function $\tau:U\to\R$ defined in a neighbourhood $U$ of $p$ such that $\tau\geq \varphi$ on $U$ and $\tau(p) = \varphi(p)$ (any such function $\tau$ is called a {\bf testing function for $\varphi$ at $p$}), one has $F_m[\chi+i\partial\bar{\partial}\tau](p)\geq g(p)$.

    In this case, we write $\fmp\geq g$ in the viscosity sense.

\end{Def}

We shall need the following simple fact:

\begin{Lem}\label{Lem:visc-subsolutions} Let $\varphi\in\pm$ and $\rho\in\pm\cap C^2(X,\,\R)$. Then, for any $\beta>1$ we have: \begin{eqnarray*}F_m\bigg[\chi+i\partial\bar{\partial}\Big((1-\frac1{\beta})\varphi+\frac1{\beta}\rho\Big)\bigg]\geq \frac1{\beta}F_m[\chi+i\partial\bar{\partial}\rho]\end{eqnarray*} in the viscosity sense.

\end{Lem}

\begin{proof} Fix $p\in X$ and let $\tau:U\longrightarrow\R$ be a $C^2$ function defined in a neighbourhood of $p$ such that $\tau\geq \varphi$ on $U$ and $\tau(p)= \varphi(p)$ (if no such $\tau$ exists, there is nothing to prove). The equivalence of distributional and viscosity subsolutions - see [HL13] for the details  implies that $\chi+i\partial\bar{\partial}\tau\geq_{m,\ \omega}0$ at $p$. Additionally, $(1-\frac1{\beta})\tau+\frac1{\beta}\rho$ is a testing function for $(1-\frac1{\beta})\varphi+\frac1{\beta}\rho$ at $p$ and any such testing function can be written in this way. (Here we use the fact that $\rho$ is $C^2$ smooth). Hence: \begin{eqnarray*} F_m\bigg[\chi+i\partial\bar{\partial}\Big((1-\frac1{\beta})\tau+\frac1{\beta}\rho\Big)\bigg](p) & = & F_m\bigg[(1-\frac1{\beta})\Big(\chi+i\partial\bar{\partial}\tau\Big)+\frac1{\beta}\Big(\chi+i\partial\bar{\partial}\rho\Big)\bigg](p) \\
    & \geq & (1-\frac1{\beta})\,F_m[\chi+i\partial\bar{\partial}\tau](p)+\frac1{\beta}\,F_m[\chi+i\partial\bar{\partial}\rho](p), \end{eqnarray*} where we used the {\it concavity} of $F_m$ restricted to $\pm\cap C^2(X)$  -- see Assumption 1.1 and the ensuing discussion in [ChX25].

  Now, $\chi+i\partial\bar{\partial}\tau\geq_{m,\ \omega}0$ at $p$ implies $F_m[\chi+i\partial\bar{\partial}\tau](p)\geq 0$.  Thus, $$F_m\bigg[\chi+i\partial\bar{\partial}\Big((1-\frac1{\beta})\tau+\frac1{\beta}\rho\Big)\bigg](p)\geq \frac1{\beta}F_m[\chi+i\partial\bar{\partial}\rho](p)$$ for any testing function for $(1-\frac1{\beta})\varphi+\frac1{\beta}\rho$ at $p$. This concludes the proof.

\end{proof}

Fix now any $\varphi\in\pm$. Without loss of generality we may assume that $sup_X\varphi\leq-2$.  As $\varphi$ is upper semicontinuous, there is a non-increasing sequence $(f_j)_j$ of $C^\infty$ functions such that $sup_Xf_j\leq -1$ for every $j$ and $f_j\searrow\varphi$ pointwise on $X$ as $j\to\infty$.  The following lemma will be crucial:

\begin{Lem}\label{Lem:almost-approximation}  For any sequence $(f_j)_j$ as above, there exist real numbers $\beta(j)>e$ and functions $\widetilde{\varphi}_j\in\pm\cap C^{\infty}(X,\,\R)$ such that $\beta(j)\nearrow\infty$ as $j\to\infty $ and $$\varphi\leq\bigg(1-\frac1{\beta(j)}\bigg)\,\varphi< \widetilde{\varphi}_j-\frac2{\beta(j)}\,\log\bigg(\frac1{\beta(j)}\bigg)\leq f_j-\frac2{\beta(j)}\,\log\bigg(\frac1{\beta(j)}\bigg).$$

\end{Lem}

Assume the lemma for a while. Then, the non-increasing pointwise approximation can be achieved as follows.

Take $\varphi_1(z):=\widetilde{\varphi}_1(z)-\frac2{\beta(1)}\,\log(\frac1{\beta(1)})$. As the sequence  $-\frac2{\beta(j)}\,\log(\frac1{\beta(j)})$ decreases to $0$, the functions $f_j-\frac2{\beta(j)}\,\log(\frac1{\beta(j)})$ decrease to $\varphi$ as $j\to\infty$. Thus, by compactness of $X$, there is an index $j=j(1)$ such that $f_j(z)-\frac2{\beta(j)}\,\log(\frac1{\beta(j)})\leq \varphi_1(z)$ for every $z\in X$. Then, define $$\varphi_2(z):=\widetilde{\varphi}_{j(1)}(z)-\frac2{\beta(j(1))}\,\log\bigg(\frac1{\beta(j(1))}\bigg).$$

By construction, we have $\varphi_2\leq f_{j(1)}-\frac2{\beta(j(1))}\,\log(\frac1{\beta(j(1))})\leq \varphi_1(z)$. Therefore, there is an index $j(2)>j(1)$ such that $f_{j(2)}-\frac2{\beta(j(2))}\,\log(\frac1{\beta(j(2))})\leq \varphi_2(z)$ and we iterate the procedure. Obviously $\varphi_j\searrow\varphi$ pointwise on $X$ as $j\to\infty$.

\vspace{1ex}

Thus, assuming that Lemma \ref{Lem:almost-approximation} has been proved, we obtain the following global approximation result for almost $m$-psh functions with respect to $\omega$.

\begin{The}\label{The:global-approximation} Let $(X,\,\omega)$ be a compact complex Hermitian manifold with $\mbox{dim}_\C X = n$. Fix $m\in\{1,\dots , n\}$ and suppose there exists a $C^\infty$ real $(1,\,1)$-form $\chi$ on $X$ such that $d\chi = 0$ and $\chi>_{m,\,\omega}0$.

For any upper semi-continuous $L^1_{loc}$ function $\varphi:X\longrightarrow\R\cup\{-\infty\}$ such that $\chi + i\partial\bar\partial\varphi\geq_{m,\,\omega}0$ on $X$, there exists a sequence $(\varphi_j)_{j\geq 1}$ of $C^\infty$ functions $\varphi_j:X\longrightarrow\R$ such that:

\vspace{1ex}

(i)\, $\chi + i\partial\bar\partial\varphi_j\geq_{m,\,\omega}0$ on $X$ for every $j\geq 1$;

\vspace{1ex}

(ii)\, for every $x\in X$, the sequence $(\varphi_j(x))_{j\geq 1}$ of reals is non-increasing and converges to $\varphi(x)$ as $j\to\infty$.

\end{The}

\vspace{2ex}

Thus, it remains to prove the above key lemma.

\begin{proof}[Proof of Lemma \ref{Lem:almost-approximation}] Fix $\varphi\in\pm$ and the non-increasing sequence of $C^\infty$ approximants $f_j$. Define \begin{equation}\label{eqn:F_m-plus}
F_m^+[\chi+i\partial\bar{\partial}f_j](z):=\begin{cases}
	F_m[\chi+i\partial\bar{\partial}f_j](z)\ \ &{\rm if}\ \ \chi+i\partial\bar{\partial}f_j(z)\geq_{m,\, \omega}0;\\
	0&{\rm otherwise}.
\end{cases}
\end{equation}

Note that for each $j$, $F_m^+[\chi+i\partial\bar{\partial}f_j]:X\longrightarrow[0,\,\infty)$ is a continuous function. Hence, there exists a $C^\infty$ function $F_m^j:X\longrightarrow\R$ such that $$0\leq F_m^+[\chi+i\partial\bar{\partial}f_j]+\frac12\leq F_m^j\leq F_m^+[\chi+i\partial\bar{\partial}f_j]+1.$$

For any parameter $\beta>e$, let $u_j^{\beta}\in \pm\cap C^{\infty}(X)$ be the (unique) solution to the equation \begin{equation}\label{eqn:Berman}
		F_m[\chi+i\partial\bar{\partial}u_j^{\beta}](z)=e^{\beta(u_j^{\beta}(z)-f_j(z))}F_m^j(z) + \frac{F_{m}[\chi](z)}{2\beta}=:G(z,u_j^{\beta}(z)).
\end{equation}

As the right-hand-side function $X\times\mathbb R\ni(z,t)\longmapsto G(z,t)$ is smooth, positive and increasing in the second variable, Theorem 3.1  from [ChX25] ensures the existence, uniqueness and smoothness of $u_j^{\beta}$. Note that here we are using assumption (\ref{alpha-m-positive}) and the property $F_m^j>0$ to ensure the strict positivity of $G$ and its increasing monotonicity in the second variable.


\vspace{1ex}

We will now show that $\widetilde{\varphi}_j$ can be taken of the form $u_j^{\beta(j)}$ for sufficiently large $\beta(j)$. To this end, let us first show that $u_j^{\beta}\leq f_j$ for any $\beta$ -- this will yield the third inequality claimed in Lemma \ref{Lem:almost-approximation}.

\vspace{1ex}

Fix a point $p\in X$ where $u_j^{\beta}-f_j$ achieves its maximum. In particular, $$(\chi+i\partial\bar{\partial}f_j )(p)\geq (\chi+i\partial\bar{\partial}u_j^{\beta})(p)$$ as $(1,1)$-forms as the Hessian of $f_j-u_j^{\beta}$ is positive semi-definite. This, coupled with $\chi+i\partial\bar{\partial}u_j^{\beta}>_{m,\, \omega}0$ (the inequality is strict since $u_j^{\beta}$ is smooth and $F_m[\chi+i\partial\bar{\partial}u_j^{\beta}]>0$) implies that \begin{eqnarray*}F_m^+[\chi+i\partial\bar{\partial}f_j](p) & = & F_m[\chi+i\partial\bar{\partial}f_j](p)\geq F_m[\chi+i\partial\bar{\partial}u_j^{\beta}](p)\geq e^{\beta(u_j^{\beta}(p)-f_j(p))}F_m^j(p)\\
 & \geq & e^{\beta(u_j^{\beta}(p)-f_j(p))}F_m^+[\chi+i\partial\bar{\partial}f_j](p).\end{eqnarray*}

This string of inequalities coupled with the positivity of $F_m^+[\chi+i\partial\bar{\partial}f_j](p)$ leads to $$0\geq \beta(u_j^{\beta}(p)-f_j(p)).$$ Thus, $u_j^{\beta}\leq f_j$ everywhere on $X$.

\vspace{1ex}

To prove the second inequality in Lemma \ref{Lem:almost-approximation}, we will use a similar idea for $\beta$ sufficiently large. Fix a point $q\in X$ where the function $$u_j^{\beta}-\bigg(1-\frac1{\beta}\bigg)\varphi-\frac2\beta\,\log\bigg(\frac1{\beta}\bigg)$$ achieves its minimum equal to $-\frac{c}{\beta}$ for some constant $c$. (A minimum does exist since $\varphi$ is upper semi-continuous and $X$ is compact.) Equivalently, $u_j^{\beta}-\frac2\beta log(\frac1{\beta})$ is a testing function at $q$ for  $(1-\frac1{\beta})\varphi+\frac c\beta.$

From Lemma \ref{Lem:visc-subsolutions} (and its proof), we get the (first) inequality below: \begin{equation}\label{eqn:lhs-in-Lemma2.4first}
  \frac1\beta F_m[\chi](q)\leq F_m[\chi+i\partial\bar{\partial}u_j^{\beta}](q) = e^{\beta(u_j^{\beta}(q)-f_j(q))}F_m^j(q)+\frac{F_m[\chi](q)}{2\beta}. \end{equation}

Thus,

\begin{equation}\label{eqn:lhs-in-Lemma2.4}
 \frac1\beta\leq e^{\beta(u_j^{\beta}(q)-f_j(q))}\frac{2F_m^j(q)}{F_m[\chi](q)}.
\end{equation}


Let us denote by $C_j$ the constant $sup_Xlog\Big(\frac{2F_m^j(z)}{F_m[\chi](z)}\Big)$. Then (\ref{eqn:lhs-in-Lemma2.4}) can be rewritten as $$e^{C_j+\beta(u_j^{\beta}(q)-f_j(q))}\geq \frac1{\beta}=e^{log(\frac1{\beta})},$$

\noindent or, more succinctly, as

\begin{equation}\label{eqn:u_j-f_j}
\frac{C_j}{\beta}+(u_j^{\beta}(q)-f_j(q))-\frac1{\beta}\,\log\bigg(\frac1{\beta}\bigg)\geq 0.
\end{equation}

Hence: \begin{align*}
&u_j^{\beta}(q)-\bigg(1-\frac1{\beta}\bigg)\,\varphi(q)-\frac2\beta\, \log\bigg(\frac1{\beta}\bigg)\geq u_j^{\beta}(q)-\bigg(1-\frac1{\beta}\bigg)\,f_j(q)-\frac2\beta\, \log\bigg(\frac1{\beta}\bigg)\\
&=\frac1{\beta}f_j(q)+(u_j^{\beta}(q)-f_j(q))-\frac2\beta\, \log\bigg(\frac1{\beta}\bigg)\geq \frac1{\beta}f_j(q)-\frac{C_j}{\beta}-\frac1\beta\, \log\bigg(\frac1{\beta}\bigg),
\end{align*} where we used (\ref{eqn:u_j-f_j}) in the last inequality. Choosing $\beta$ so large that $log(\beta)> -inf_Xf_j+C_j$ yields $$u_j^{\beta}-\bigg(1-\frac1{\beta}\bigg)\varphi-\frac2\beta\, \log\bigg(\frac1{\beta}\bigg)>0$$ at $q$, hence everywhere. This completes the proof.

\end{proof}

\subsection{Local regularisation}\label{subsection:local-regularisation}

In this subsection, we consider the local regularisation problem. In particular, $X$ may not be compact. We consider a $(1,\,1)$-current $T=\chi+i\partial\bar\partial\varphi\geq_{m,\, \omega}0$ on $X$, we fix an arbitrary point $p\in X$ and we seek a non-increasing sequence $\varphi_j\searrow\varphi$ of $C^\infty$ functions $\varphi_j:U_p\longrightarrow\R$ defined in an open neighbourhood $U_p\subset X$ of $p$.

To this end, we fix a coordinate chart centred at $p$ and a sufficiently small open ball $U_p$ in these coordinates. Note that $\chi=i\partial\bar\partial\rho$ for some $C^\infty$ function $\rho:U_p\longrightarrow\R$ (shrink $U_p$ about $p$ if necessary). Hence, the function $u(z):=\rho(z)+\varphi(z)$ defined for $z\in U_p$ is $m$-psh with respect to $\omega$ (cf. definition (\ref{eqn::m-semi-pos_equivalence-potential})) and $T_{|U_p} = i\partial\bar\partial u$. Thus, the regularisation problem reduces to non-increasingly approximating $u$ by $m$-psh (with respect to $\omega$) functions $u_j$ defined on $U_p$.

Much as in the argument in the previous subsection, we let $(f_j)_{j\geq 1}$ be a sequence of $C^\infty$ functions $f_j:U_p\longrightarrow\R$ decreasing pointwise to $u$. By adding a constant to $\rho$, if necessary, we may assume that $\sup_{U_p}u\leq -2$ and $\sup_{U_p}f_j\leq -1$ for every $j\geq 1$.

\vspace{1ex}

The local analogue of the operator $F_m$ of Definition \ref{Def:F_m} is given in the following

\begin{Def}\label{Def:local_F_m} Let $U_p\subset X$ be as above. For any $C^2$ function $v:U_p\longrightarrow\R$ such that $i\partial\bar\partial v\geq_{m,\, \omega}0$ on $U_p$, let $\lambda_1\leq\dots\leq\lambda_n$ be the eigenvalues of the $(1,\,1)$-form $i\partial\bar\partial v$ with respect to $\omega$ -- see (\ref{eqn:coordinates_diagonalisation}). Then, the function $F_m[i\partial\bar{\partial}v]$ is defined at every point $z\in U_p$ by the formula:

\begin{equation*}F_m[i\partial\bar{\partial}v](z):=\Big(\Pi_{|J|=m}\Big[\sum_{j\in J}\lambda_j(z)\Big]\Big)^{1/\binom nm},\end{equation*} where $J=(j_1,\cdots,j_m)$ runs over all ordered $m$-tuples of indices from $\lbrace1,\cdots,n\rbrace$.

\end{Def}

Just as in the previous subsection, one can define viscosity subsolutions through testing with local $C^2$ functions touching the would-be subsolution from above.

The analogue of Lemma \ref{Lem:visc-subsolutions} can be stated as follows:

\begin{Lem}\label{Lem:local-visc-subsolutions} Let $v\in C^2(U_p,\,\R)$ and $u:U_p\longrightarrow\R\cup\{-\infty\}$ be $m$-psh functions with respect to $\omega$. Then, for any $\beta>1$ we have:

$$F_m\bigg[i\partial\bar{\partial}\Big(u+\frac1{\beta}v\Big)\bigg]\geq \frac1{\beta}F_m[i\partial\bar{\partial}v]$$ in the viscosity sense at every point $z\in U_p$.

\end{Lem}

\begin{proof}

 As the proof repeats the one of Lemma \ref{Lem:visc-subsolutions} from the global case, we skip the details.

\end{proof}

Just as in the global case, we shall need a result on the smooth solvability of the Dirichlet problem for the local $F_m$ operator.  Despite the extensive literature on the subject (see [GGQ22], [Do23] and the references therein), we were unable to find the exact statement, spelt out below, that is relevant to our case. Therefore, we shall provide a proof in the Appendix ($\S$\ref{section:Appendix}).

\begin{The}\label{Thm:Dirichlet-Fm} Let $(X,\omega)$ be an $n$-dimensional complex Hermitian manifold, $m\in\{1,\dots , n\}$, $p\in X$ an arbitrary point and $U_p\subset X$ a sufficiently small ball in a coordinate chart about $p$. Then, for any constant $\beta>e$, the Dirichlet problem:

 \begin{equation*}
  \begin{cases}
  u_j^{\beta}\in C^{\infty}(\overline U_p);&\\
  i\partial\bar{\partial}u_j^{\beta}>_{m,\, \omega}0;&\\
   F_m[i\partial\bar{\partial}u_j^{\beta}]=e^{\beta(u_j^{\beta}-f_j)}+\frac{1}{2\beta}\ &{\rm in}\ U_p;\\
   u_j^\beta|_{\partial U_p}=f_j|_{\partial U_p}
  \end{cases}
 \end{equation*} admits a unique solution $u_j^{\beta}$.
\end{The}

We are now ready to state and prove the main result of this subsection:

\begin{The}\label{The:local-approximation}

 If $u:U_p\longrightarrow\R\cup\{-\infty\}$ is any $m$-psh (with respect to $\omega$) function, there exists a sequence $(u_j)_{j\geq 1}$ of $C^\infty$, strictly $m$-psh (with respect to $\omega$) functions $u_j:U_p\longrightarrow\R$ such that the sequence $(u_j(x))_{j\geq 1}$ of reals is non-increasing and converges to $u(x)$ for every point $x\in X$.

\end{The}

\begin{proof}

  Let $(f_j)_{j\geq 1}$ be the approximating sequence of $C^\infty$ functions $f_j:U_p\longrightarrow\R$ decreasing pointwise to $u$ that were fixed above and let $(u_j^{\beta})_{j\geq 1}$ be the corresponding solutions to the Dirichlet problem provided by Theorem \ref{Thm:Dirichlet-Fm}.

 \vspace{1ex}

 {\bf Claim $1$.} {\it If $c_j:=max\lbrace sup_{U_p}F_m[i\partial\bar{\partial}f_j],e\rbrace$, then for any $\beta>e$ we have:

 \begin{equation}\label{eqn:uleqfj}
  u_j^{\beta}\leq f_j+\frac{log(c_j)}{\beta} \hspace{5ex}\mbox{on}\hspace{1ex} \overline{U}_p, \hspace{2ex} j\geq 1.
 \end{equation}}

 \noindent {\it Proof of Claim $1$.} Let $q$ be a point at which $u_j^{\beta}-f_j$ achieves its maximum over $\overline{U}_p$. If $q\in\partial U_p$, then $u_j^{\beta}(q)-f_j(q) = 0$ by the boundary condition and we are done. Otherwise, we have:

 $$i\partial\bar{\partial}u_j^{\beta}(q)\leq i\partial\bar{\partial}f_j(q),$$

 hence \begin{align*}
  c_j\geq F_m[i\partial\bar{\partial}f_j](q)\geq F_m[i\partial\bar{\partial}u_j^{\beta}](q)=e^{\beta(u_j^{\beta}(q)-f_j(q))}+\frac{1}{\beta}\geq e^{\beta(u_j^{\beta}(q)-f_j(q))},&
 \end{align*}

which yields the claim.   \hfill $\Box$

\vspace{1ex}

{\bf Claim $2$.} {\it There exists a constant $C>0$, depending only on $(X,\omega, U_p)$, such that $$u\leq u_{j}^{\beta}+C\frac{r^2}{\beta}-\frac1\beta log(\frac1{2\beta})$$ at every point of $\overline{U}_p$, where $r$ is the radius of $U_p$.}

\vspace{1ex}

\noindent {\it Proof of Claim $2$.} We consider the function $$l(z):=u_j^{\beta}(z)-u(z)-\frac{C}{\beta}(||z||^2-r^2)$$ defined on $\overline{U}_p$, with the constant $C$ chosen such that, at every point of $\overline{U}_p$, we have: $$CF_m[i\partial\bar{\partial}(||z||^2-r^2)]\geq 1.$$ (The existence of such a $C$ follows from the strict plurisubharmonicity of $||z||^2$).

If the minimum of this function over $\overline{U}_p$ occurs at a boundary point $q$, the relations $u_j^{\beta}(q)=f_j(q)\geq u(q)$ conclude the argument. Henceforth we assume that the minimum is achieved at some point $q\in U_p$. Thus, $u_j^{\beta}-min_{U_p}l$ is a testing function for $u+\frac{C}{\beta}(||z|^2-r^2)$ at $q$. Using the subsolution property of $u$  and  Lemma \ref{Lem:local-visc-subsolutions} we obtain

$$F_m[i\partial\bar{\partial}u_j^{\beta}](q)\geq\frac C{\beta}F_m[i\partial\bar{\partial}(||z||^2-r^2)](q)\geq \frac{1}{\beta}.$$

Coupling this with the definition of $u_j^{\beta}$ we obtain

\begin{equation*}
 e^{\beta(u_j^{\beta}(q)-f_j(q))}+\frac{1}{2\beta}\geq\frac1{\beta},
\end{equation*}

which in turn implies $\beta(u_j^{\beta}(q)-f_j(q))\geq log(\frac1{2\beta})$.  Hence:

\begin{align*}
 u_j^{\beta}(z)+C\frac{r^2}{\beta}-\frac1\beta log(\frac1{2\beta})-u(z)\geq& u_j^{\beta}(z)-u(z)-\frac{C}{\beta}(||z||^2-r^2)-\frac1\beta log(\frac1{2\beta}),  \hspace{5ex} z\in\overline{U}_p,\\
 u_j^{\beta}(q)-u(q)-\frac{C}{\beta}(||q||^2-r^2)-\frac1\beta log(\frac1{2\beta})\geq& f_j(q)-u(q)-\frac{C}{\beta}(||q||^2-r^2)\geq 0+0=0.&
\end{align*}

This completes the proof of Claim $2$. \hfill $\Box$

\vspace{2ex}

\noindent {\it End of proof of Theorem \ref{The:local-approximation}.} With each constant $c_j$ as in Claim 1 we associate a sufficiently large constant $\beta(j)$ such that the sequence $\frac{log(c_j)}{\beta(j)}$ decreases to zero as $j\to\infty$. Then, $C>0$ being the constant given by Claim $2$, we consider the sequence of functions defined for $z\in\overline{U}_p$:

\begin{equation}\label{eqn:uj}
 \widehat{u}_j(z):=u_j^{\beta(j)}(z)+2C\frac{r^2}{\beta(j)}-\frac1{\beta(j)} log(\frac1{2\beta(j)}),  \hspace{5ex} j\geq 1.
\end{equation}

(These functions are $C^\infty$ and strictly $m$-psh with respect to $\omega$).

From Claim $2$, we get the first inequality below: $$u(z)\leq u_j^{\beta(j)}(z)+C\frac{r^2}{\beta(j)}-\frac1{\beta(j)} log(\frac1{2\beta(j)})< \widehat{u}_j(z),$$

\noindent and from Claim $1$ we get: $$\widehat{u}_j(z)\leq f_j(z)+2C\,\frac{r^2}{\beta(j)}+\frac{log(c_j)}{\beta(j)}-\frac1{\beta(j)} log(\frac1{2\beta(j)}).$$

Note that

\begin{equation}\label{eqn:fjtou}
 f_j(z)+{2}C\,\frac{r^2}{\beta(j)}+\frac{log(c_j)}{\beta(j)}-\frac1{\beta(j)} log(\frac1{2\beta(j)})\searrow u(z)\ {\rm over}\ \overline{U}_p.
\end{equation}

 Take now $u_1:=\hat{u}_1$ and $j=j(1)$ so large so that $f_j(z)+{2}C\frac{r^2}{\beta(j)}+\frac{log(c_j)}{\beta(j)}-\frac1{\beta(j)} log(\frac1{2\beta(j)})\leq u_1$ (which is possible since the convergence in (\ref{eqn:fjtou}) is over the {\it compact} set $\overline{U}_p$). Then define $u_2:=\widehat{u}_{j(1)}$, which is, by construction, less than $u_1$. Take now $j(2)>\beta(j)$ such that $f_j(z)+{2}C\frac{r^2}{\beta(j)}+\frac{log(c_j)}{\beta(j)}-\frac1{\beta(j)} log(\frac1{2\beta(j)})\leq u_2$ and proceed in the same fashion to construct the decreasing sequence $\lbrace u_j\rbrace_{j=1}^{\infty}$. This completes the proof of Theorem \ref{The:local-approximation}.

\end{proof}


\section{Appendix}\label{section:Appendix} The goal of this section is to spell out the proof of Theorem \ref{Thm:Dirichlet-Fm}, needed for the local smoothing of $m$-subharmonic functions.

  \begin{The}

   Let $U$ be a sufficiently small coordinate ball in a coordinate chart of $(X,\omega)$. Let $f$ be a smooth function in a neighbourhood of $U$. Finally, let

   $$G: \overline U\times\mathbb R\ni (z,t)\longmapsto G(z,t)\in \mathbb R$$

   be a $C^\infty$ positive function satisfying $\frac{\partial G}{\partial t}(z,t)>0.$

   Then, the Dirichlet problem:

   \begin{equation}\label{thm:local_Dirichlet_problem}
    \begin{cases}
      u\in C^{\infty}(\overline{U});\\
      F_m[i\partial\bar{\partial}u](z)=G(z,u(z));\\
      u|_{\partial U}=f|_{\partial U}
     \end{cases}
   \end{equation}

admits a unique $m$-subharmonic solution.

  \end{The}

\begin{proof}

 The uniqueness is an easy consequence of the maximum principle.

 To show existence, we shall employ the standard continuity method. Let $r$ be the radius of the ball and let $C=C(\omega,f)>0$ be a constant such that $C(||z||^2-r^2)+f$ is strictly $m$-subharmonic near $\overline{U}$. Consider the equations:

 \begin{equation}\label{eqn:Dirichlet-path}
  (*_t)\ \ \begin{cases}
          u_t\in C^{\infty}(\overline{U});\\
      F_m[i\partial\bar{\partial}u_t](z)=tG(z,u_t(z))+(1-t)F_m[i\partial\bar{\partial}\Big(C(||z||^2-r^2)+f\Big)];\\
      u_t|_{\partial U}=f|_{\partial U}.
         \end{cases}
 \end{equation}

It is clear that $u_0(z):=C(||z||^2-r^2)+f$ solves $(*_0)$. It suffices to show that the set

$$A:=\lbrace t\in[0,1]\ |\ (*_t)\ {\rm is\ solvable}\rbrace$$

is open and closed.

\vspace{1ex}

{\bf Openness}: Assume $(*_{t_0})$ is smoothly solvable for some $t_0$. The linearisation operator $L_{t_0}$ for $(*_{t_0})$ is then a uniformly elliptic linear second-order operator without lower-order terms. Standard elliptic theory ensures unique solvability for the problem:

$$\begin{cases}

   L_{t_0}(v(z))=H(z,v(z));\\

   v|_{\partial U}=f|_{\partial U}

  \end{cases}
$$

for all $H$ sufficiently close to $L_{t_0}(u_{t_0})$ in $C^{\alpha}$ norm, $\alpha\in(0,1)$. By the implicit function theorem, $(*_{t})$ is then solvable for $t$ sufficiently close to $t_0$ and the solution is $C^{2,\alpha}$. As the right-hand side of (\ref{eqn:Dirichlet-path}) is then $C^{2,\alpha}$, $u_t$ must be $C^{4,\alpha}$ by elliptic regularity. By bootstrapping, $u_t$ is in fact $C^{\infty}$.

\vspace{1ex}

{\bf Closedness}: As is well known - see [Yau78], [GN15], [ChX25], [GN18] or [Do23], it suffices to prove estimates on the $C^2$-norm of the solution to the problem $(*_t)$. In what follows, we follow closely the argument of [GN18].

Our approach relies crucially on the following concept from linear algebra: if $g$ and $g_0$ are Hermitian inner products on a finite-dimensional $\C$-vector space $V$ in which a basis has been fixed, a matrix $A$ is said to be $g$-Hermitian if, for any vectors $X,\ Y\in V$, we have:

\begin{equation}\label{eqn:gHermitian}
g(AX,Y)=g(X,AY).
\end{equation}

It is easily checked that a $g$-Hermitian matrix $A$ can be realised as $(g)^{-1}\tilde{A}$ for a unique $g_0$-Hermitian matrix $\tilde{A}$, where $(g)^{-1}$ denotes the inverse of the transpose of the matrix $(g_{p\bar{q}})$.


Abusing the notation slightly, we shall define the action of $F_m[A]$ on a $g$-Hermitian matrix $A$ by identifying $A=(g)^{-1}(\tilde{a}_{j\bar{k}})$ with its associated real $(1,1)$ form $\tilde{a}:=i\sum_{j,k=1}^n\tilde{a}_{j\bar{k}}dz_j\wedge d\bar{z}_k$ and putting $F_m[A]:=F_m[\tilde{a}]$.

We shall make heavy use of the derivatives of $F_m$ at a diagonal Hermitian matrix $A$. We recall the following formula, whose proof consists in a standard computation:

\begin{Lem}\label{lem:F_m}

 Let $A$ be a diagonal Hermitian matrix with eigenvalues (with respect to $Id$) $\lambda_j$,\ $j=1,\cdots,n$. Then

 $$F^{p\bar{q}}_m[A]:=\frac{\partial F_m[A]}{\partial{a_{p\bar{q}}}}=\begin{cases}

                       0\ \ {\rm if}\ p\neq q;\\

                       \frac1{\binom{n}{m}}F_m[A]\sum_{p\in{\lbrace i_1,\cdots ,i_m}\rbrace}\frac{1}{\lambda_{i_1}+\cdots+\lambda_{i_m}}\ {\rm if}\ p=q.

                      \end{cases}
$$

\end{Lem}

A simple consequence of the above formula is the following observation:

\begin{Lem}\label{lem:ZZ}

 Let $A$ be a diagonal Hermitian matrix with eigenvalues $\lambda_j$,\ $j=1,\cdots,n$. Then, for a universal constant $\nu$ depending only on $n$ and $m$, we have:

 $$\Pi_{p=1}^nF^{p\bar{p}}_m[A]\geq \nu.$$

\end{Lem}

\begin{proof}

This follows from the application of AM-GM inequality.

\end{proof}

Next, we establish the required a priori bounds on the solution.

\vspace{1ex}

{\bf Step 1: uniform estimate}. As the solution $u$ is, in particular, subharmonic we have

$$sup_{\overline{U}}u=sup_{\partial{U}}u=sup_{\partial{U}}f.$$

Monotonicity of $G$ then implies that $F_m[i\partial\bar{\partial}u](z)\leq G(z,max_{\partial U}f)\leq C\leq F[i\partial\bar{\partial}\big(a(||z||^2-r^2)+f\big)](z)$

once the positive constant $a$ is taken to be sufficiently large. The maximum principle then provides the lower bound of $u$ by the infimum of $\big(a(||z||^2-r^2)+f\big)(z)$. Note also that, increasing $a$ further if necessary, $-a(||z||^2-r^2)+f$ would become a superharmonic function.

\vspace{1ex}

{\bf Step 2: gradient estimate}.

Note that the inequalities

$$\big(a(||z||^2-r^2)+f\big)(z)\leq u(z)\leq \big(-a(||z||^2-r^2)+f\big)(z)$$

immediately yield a gradient bound on the boundary of $U$. Hence we are left with proving the bound on the gradient in the interior of $\overline{U}$.

Let $\omega=i\sum_{j,k=1}^ng_{j\bar{k}}dz_j\wedge d\bar{z}_k$ around $\overline{U}$. We have

\begin{equation}\label{eqn:c1}
 \frac{\delta_{jk}}{C_1}\leq g_{j\bar{k}}\leq C_1\delta_{jk}
\end{equation}

for a constant $C_1>0$ depending only on $\omega$. Notice that, choosing $r$ sufficiently small, we can arrange for $||z||_g^2:=g_{j\bar{k}}(z)z_j\bar{z}_k$ to have the following property (cf. [GN18]):

 \begin{equation}\label{eqn:normz}
 \forall p\in\lbrace1,\cdots, n\rbrace,\ \ \ \frac{\partial^2}{\partial z_p\partial\bar{z}_p}(||z||_g^2)=g_{p\bar{p}}(z)+O(|z|)\geq  g_{p\bar{p}}/2.
\end{equation}

Following [GN18], we consider the function

\begin{equation}\label{eqn:Hfunction}
 H(z):=log(||\nabla u||_{\omega}^2(z))+\psi(u(z)+v(z)),
\end{equation}

where $||\nabla u||_{\omega}^2:=\Lambda_{\omega}(i\partial u \wedge \bar{\partial} u)$ which in local coordinates can be written as $g^{j\bar{k}}u_ju_{\bar{k}}$ with $u_j:=\frac{\partial u}{\partial z_j}$. Furthermore $v(z):=N(sup_{\overline{U}}||z||_g^2-||z||_g^2)$ for a constant $N>0$ to be chosen later on. Finally, let

\begin{equation}\label{eqn:psi}
\psi(t):=-\frac12log\bigg(1+\frac{t}{L+NC_1r^2}\bigg),
\end{equation}

where $L:=1+sup_{U}|u|$. Note that

$\psi$ is chosen such that $\psi'<0$ and $\psi''=2\psi'^2$.

\vspace{1ex}

Suppose now that $H$ achieves its maximum at an interior point $z_0$. After a complex linear change of coordinates, we can arrange for $g_{j\overline{k}}$ to be the identity matrix at $z_0$ and for $(u)_{j\overline{k}}$ to be diagonal there. Recall the following standard formula and inequality for the first and second derivatives of $H$ at $z_0$: for every $p\in\{1,\dots,n\}$, one has: \begin{equation}\label{eqn:H_p} 0 = H_{p}(z_0)=\frac{g^{j\bar{k}}_pu_ju_{\bar{k}}+u_{jp}u_{\bar{j}}+u_pu_{p\bar{p}}}{||\nabla u||_{\omega}^2}+\psi'(u_p+v_p);
\end{equation}




\begin{align*}
 &0\geq H_{p\bar{p}}\geq \frac{g^{j\bar{k}}_{p\bar{p}}u_ju_{\bar{k}}+2\Re\Big(g^{j\bar{k}}_pu_ju_{\bar{p}\bar{k}}\Big)+2\Re\Big(g^{p\bar{k}}_pu_{p\bar{p}}u_{\bar{k}}\Big)+2\Re\Big(u_{jp\bar{p}}u_{\bar{j}}\Big)+|u_{jp}|^2+u_{p\bar{p}}^2}{||\nabla u||_{\omega}^2}\\
 &-|\frac{g^{j\bar{k}}_pu_ju_{\bar{k}}+u_{jp}u_{\bar{j}}+u_pu_{p\bar{p}}}{||\nabla u||_{\omega}^2}|^2+\psi'(u_{p\bar{p}}+v_{p\bar{p}})+\psi''|u_p+v_p|^2.
\end{align*}

We now estimate $\sum_pF_m^{p\bar{p}}H_{p\bar{p}}$ from below at $z_0$. Here and afterwards, $F_m^{p\bar{p}}$ denotes $\frac{\partial F_m[D^2_{\mathbb C}u]}{\partial{a_{p\bar{q}}}}(z_0)$, with $D^2_{\mathbb C}u$ denoting the complex Hessian of $u$. To this end, we exploit the equation

$$\sum_pF_m^{p\bar{p}}u_{p\bar{p}k}=G(z,u(z))_k$$

(valid at $z_0$) and (\ref{eqn:H_p}) to simplify the second inequality below:

\begin{equation}\label{eqn:initial_simplification}
 0\geq \sum_pF_m^{p\bar{p}}H_{p\bar{p}}\geq
\end{equation}
\begin{align*}
 &\frac{F_m^{p\bar{p}}g^{j\bar{k}}_{p\bar{p}}u_ju_{\bar{k}}+2F_m^{p\bar{p}}\Re\Big(g^{j\bar{k}}_pu_ju_{\bar{p}\bar{k}}\Big)+2F_m^{p\bar{p}}\Re\Big(g^{p\bar{k}}_pu_{p\bar{p}}u_{\bar{k}}\Big)+2\Re\Big(G_ju_{\bar{j}}\Big)+F_m^{p\bar{p}}|u_{jp}|^2+F_m^{p\bar{p}}u_{p\bar{p}}^2}{||\nabla u||_{\omega}^2}\\
 &+\Big(\psi''-(\psi')^2\Big)F_m^{p\bar{p}}|u_p+v_p|^2+\psi'\Big(G+\sum_pF_m^{p\bar{p}}v_{p\bar{p}}\Big)=:A+B+C.
\end{align*}

$\bullet$ We start by establishing a lower bound for $A$. Note that for a geometric constant $R$, one has:

\begin{equation}\label{eqn:R}
 \frac{F_m^{p\bar{p}}g^{j\bar{k}}_{p\bar{p}}u_ju_{\bar{k}}}{||\nabla u||_{\omega}^2}\geq-R\frac{F_m^{p\bar{p}}||\nabla u||_{\omega}^2}{||\nabla u||_{\omega}^2}\geq -R\mathcal F,
\end{equation}

where we set $\mathcal F:=\sum_pF_m^{p\bar{p}}$ at $z_0$. We also remark that the constant $R$ is related to the curvature bound of $X$ at $z_0$.

Analogously, for a geometric constant $T>0$, one has: \begin{equation}\label{eqn:T1grad}
 \frac{2F_m^{p\bar{p}}\Re\Big(g^{j\bar{k}}_pu_ju_{\bar{p}\bar{k}}\Big)}{||\nabla u||_{\omega}^2}\geq -\frac{F_m^{p\bar{p}}\sum_k|u_{kp}|^2}{||\nabla u||_{\omega}^2}-T\frac{F_m^{p\bar{p}}||\nabla u||_{\omega}^2}{||\nabla u||_{\omega}^2}\geq -\frac{F_m^{p\bar{p}}\sum_k|u_{kp}|^2}{||\nabla u||_{\omega}^2}-T\mathcal F
\end{equation}

and

\begin{equation}\label{eqn:T2grad}
 \frac{2F_m^{p\bar{p}}\Re\Big(g^{p\bar{k}}_pu_{p\bar{p}}u_{\bar{k}}\Big)}{||\nabla u||_{\omega}^2}\geq -\frac{F_m^{p\bar{p}}u_{p\bar{p}}^2}{||\nabla u||_{\omega}^2}-T\frac{F_m^{p\bar{p}}||\nabla u||_{\omega}^2}{||\nabla u||_{\omega}^2}\geq-\frac{F_m^{p\bar{p}}u_{p\bar{p}}^2}{||\nabla u||_{\omega}^2}-T\mathcal F.
\end{equation}

Finally note that, assuming $||\nabla u||_{\omega}^2\geq1$, we have:

\begin{equation}\label{eqn:G_j}
 \frac{2\Re\Big(G_ju_{\bar{j}}\Big)}{||\nabla u||_{\omega}^2}\geq -C_2
\end{equation}

for a constant $C_2$ depending only on the gradient bound of $G$, which we already control.

Collecting (\ref{eqn:R}), (\ref{eqn:T1grad}), (\ref{eqn:T2grad}) and (\ref{eqn:G_j}) and plugging them into (\ref{eqn:initial_simplification}), we end up with

\begin{equation}\label{eqn:estimate_for_A}
A\geq -(R+2T)\mathcal F-C_2.
\end{equation}

$\bullet$ Next, we deal with the term $B$. Recall that the choice of $\psi$ implies $\psi''=2\psi'^2$, hence

\begin{equation}\label{eqn:initial_B}
B\geq (\psi')^2 F_m^{p\bar{p}}|u_p+v_p|^2.
\end{equation}

$\bullet$ Turning to the term $C$, we note that $\psi'G$ is bounded below by a constant $C_3$ depending only on the supremum norm of $G$. Thus:

\begin{equation}\label{eqn:estimate_for_C}
 \psi'\Big(G+\sum_pF_m^{p\bar{p}}v_{p\bar{p}}\Big)\geq-C_3+|\psi'|\frac{N}{2C_1}\mathcal F,
\end{equation}

where we utilised (\ref{eqn:c1}) and (\ref{eqn:normz}).

\vspace{2ex}

Finally, putting together (\ref{eqn:estimate_for_A}), (\ref{eqn:initial_B}) and (\ref{eqn:estimate_for_C}), we get:

\begin{equation}\label{eqn:after_simplification}
0\geq \bigg(|\psi'|\frac{N}{2C_1}-R-2T\bigg)\mathcal F-C_2-C_3+(\psi')^2 F_m^{p\bar{p}}|u_p+v_p|^2.
\end{equation}

Recall now that, from the very definition of $\psi$ and $v$, we have: $u+v\in[0,L+NC_1r^2]$. Hence:

$$|\psi'\circ(u+v)|\geq\frac1{4(L+NC_1r^2)}.$$

Hence:

\begin{equation}\label{eqn:Term_in_front_of_F}
 \bigg(|\psi'|\frac{N}{2C_1}-R-2T\bigg)\geq\bigg(\frac{N}{8C_1(L+NC_1r^2)}-R-2T\bigg)\geq 1
\end{equation}

provided we first choose $r>0$ so small that

$$9C_1^2(R+2T+1)r^2<1$$

and then $N$ sufficiently large, depending on $R, T, C_1$ and $L$.

Then, using (\ref{eqn:Term_in_front_of_F}), (\ref{eqn:after_simplification}) leads to

\begin{equation}\label{eqn:What_we_want}
 C_4\geq \mathcal F+C_5F_m^{p\bar{p}}|u_p+v_p|^2
\end{equation}

for some uniform constants $C_4$ and $C_5$ depending on $L$ and the geometry of $(X,\omega)$.

The bound $\mathcal F\leq C_4$ coupled with Lemma \ref{lem:ZZ} yields, for all $p\in\{1,\cdots,n\}$:

$F_m^{p\bar{p}}\geq C_6>0$ for some universal constant $C_6$. Then,

$C_4\geq C_5F_m^{p\bar{p}}|u_p+v_p|^2$ immediately implies $|u_p+v_p|^2\leq C_7$. Finally, as $||\nabla v||$ is bounded, we obtain:

\begin{equation}\label{eqn:nabla_u}
 ||\nabla u||\leq C_8.
\end{equation}

\vspace{1ex}

{\bf Step 3: second order estimate - reduction to the boundary}.

Similarly to the previous step, we prove a bound  on the complex Hessian of $u$ at an interior point. To this end, we consider the quantity:

\begin{equation}\label{eqn:tilde_S}
\tilde{S}(z,\xi)=log(1+u_{i\bar{j}}\xi^i\bar{\xi}^j)+\varphi(||\nabla u||^2_{\omega})+\psi(u+v),
\end{equation}

where $\xi=(\xi^1,\cdots,\xi^n)$ is a unit vector from $T^{1,0}_zX$ (the quantity obviously does not depend on the choice of coordinates), $v$ is as in the previous proof, $\varphi(t):=-\frac12log(1-\frac{t}{2K})$ for $K:=sup||\nabla u||^2_{\omega}+1$ and $\psi(t)=-\frac16log(1+\frac{t}{2L})$ with $L:=1+sup|u|$.

Suppose $\tilde{S}$ achieves its maximum at a point $z_0$ in the direction $\xi_0$. Then, applying a rotation around $z_0$ if necessary, we construct coordinates $(z_1,\cdots,z_n)$ such that $u_{i\bar{j}}(z_0)$ is diagonal, $g_{i\bar{j}}(z_0)=\delta_{ij}$ and, moreover, $\xi_0$ can be extended as $g_{1\bar{1}}^{-1/2}\frac{\partial}{\partial z_1}$ in a neighbourhood. Then, the modified function:

\begin{equation}\label{eqn:notilde_S}
{S}(z)=log(1+g^{1\bar{1}}u_{1\bar{1}})+\varphi(||\nabla u||^2_{\omega})+\psi(u+v),
\end{equation}

still has a local maximum at $z_0$. Differentiating, we obtain the equality:

\begin{equation}\label{eqn:S_p}
0=S(z_0)_p=\frac{g^{1\bar{1}}_pu_{1\bar{1}}+g^{1\bar{1}}u_{1\bar{1}p}}{1+u_{1\bar{1}}} +\varphi'\big(||\nabla u||^2_{\omega}\big)_p+\psi'\big(u_p+v_p\big).
\end{equation}

Differentiating once more at $z_0$ and summing with $F^{p\bar{p}}_m$ results in

\begin{equation}\label{eqn:S_pbarp}
0\geq F^{p\bar{p}}_m \, \frac{g^{1\bar{1}}_{p\bar{p}}u_{1\bar{1}}+2\Re\big(g^{1\bar{1}}_{p}u_{1\bar{1}\bar{p}}\big)+u_{1\bar{1}p\bar{p}}}{1+u_{1\bar{1}}}-F^{p\bar{p}}_m \, \bigg|\frac{g^{1\bar{1}}_pu_{1\bar{1}}+g^{1\bar{1}}u_{1\bar{1}p}}{1+u_{1\bar{1}}}\bigg|^2
\end{equation}

$$+\varphi''F^{p\bar{p}}_m \, |\big(||\nabla u||^2_{\omega}\big)_{p}|^2+\varphi' \,F^{p\bar{p}}_m\big(||\nabla u||^2_{\omega}\big)_{p\bar{p}}$$

$$+\psi'' \,F^{p\bar{p}}_m |u_p+v_p|^2+\psi' \, F^{p\bar{p}}_m\big(u_{p\bar{p}}+v_{p\bar{p}}\big)=:I+II+III+IV+V+VI.$$

\vspace{1ex}

$\bullet$ For the term $I$, we observe that (recall that $g_{i\bar{j}}(z_0)=\delta_{ij}$) rewriting $2\Re\big(g^{1\bar{1}}_{p}u_{1\bar{1}\bar{p}}\big)$ as

$$2\Re\big(g^{1\bar{1}}_{p}[u_{1\bar{1}\bar{p}}+g^{1\bar{1}}_{p}u_{1\bar{1}}-g^{1\bar{1}}_{p}u_{1\bar{1}}]\big),$$

reordering and using the inequality $2\Re(a\bar{b})\geq - \frac12|a|^2+2|b|^2$, we end up with the bound:

\begin{equation}\label{eqn:I}
I\geq -R\mathcal F-\frac12F^{p\bar{p}}_m|\frac{g^{1\bar{1}}_pu_{1\bar{1}}+g^{1\bar{1}}u_{1\bar{1}p}}{1+u_{1\bar{1}}}|^2+F^{p\bar{p}}_m\frac{u_{1\bar{1}p\bar{p}}}{1+u_{1\bar{1}}}
\end{equation}

for a constant $R$ depending only on the geometry of $(X,\omega)$.

Next, we recall that $F_m$ is a concave operator. This follows from the same argument as in [Din23, Section 1], where the real counterpart of $F_m$ was considered. We have (assuming $u_{1\bar{1}}(z_0)\geq 1$):

$$F^{p\bar{p}}_m{u_{1\bar{1}p\bar{p}}}=G_{1\bar{1}}+F_m^{r\bar{s},q\bar{l}}u_{r\bar{s}1}u_{q\bar{l}\bar{1}}\geq -C_9u_{1\bar{1}}$$

for a constant $C_9$ depending on $G$. Here, $F_m^{r\bar{s},q\bar{l}}$ denotes the second derivative of $F_m$.

Thus, we can estimate:

\begin{equation}\label{eqn:IplusII}
 I+II\geq -R\mathcal F-\frac32F^{p\bar{p}}_m|\frac{g^{1\bar{1}}_pu_{1\bar{1}}+g^{1\bar{1}}u_{1\bar{1}p}}{1+u_{1\bar{1}}}|^2-C_9
\end{equation}

$$\geq -R\mathcal F-C_9-\frac32F^{p\bar{p}}_m|\varphi'\big(||\nabla u||^2_{\omega}\big)_p+\psi'\big(u_p+v_p\big)|$$

$$\geq -R\mathcal F-C_9-2(\varphi')^2F^{p\bar{p}}_m|\big(||\nabla u||^2_{\omega}\big)_p|^2-6(\psi')^2F^{p\bar{p}}_m|\big(u_p+v_p\big)|^2,$$

where we used (\ref{eqn:S_p}) to get the penultimate inequality.

The very choices of $\varphi$ and $\psi$ lead to the equalities: $\varphi''=2(\phi')^2$ and $\psi''=6(\psi')^2$. These, coupled with (\ref{eqn:IplusII}), imply:

\begin{equation}\label{eqnIplusIIplusIIIplusV}
I+II+III+V\geq -R\mathcal F-C_9.
\end{equation}

$\bullet$ Next, we turn our attention to the term $IV$. Just as in (\ref{eqn:H_p}), we compute at $z_0$:

\begin{equation}\label{eqn:IV}
IV=\varphi'\Big[F_m^{p\bar{p}}g^{j\bar{k}}_{p\bar{p}}u_ju_{\bar{k}}+2F_m^{p\bar{p}}\Re\Big(g^{j\bar{k}}_pu_ju_{\bar{p}\bar{k}}\Big)+2F_m^{p\bar{p}}\Re\Big(g^{p\bar{k}}_pu_{p\bar{p}}u_{\bar{k}}\Big)+2\Re\Big(G_ju_{\bar{j}}\Big)\big]
\end{equation}

$$+\varphi'\Big[F_m^{p\bar{p}}|u_{jp}|^2+F_m^{p\bar{p}}u_{p\bar{p}}^2\Big].$$

Estimating similarly to (\ref{eqn:R}), (\ref{eqn:T1}) and (\ref{eqn:T2}) and exploiting the bound on $||\nabla u||$, we end up with

\begin{equation}\label{eqn:IVfinal}
IV\geq -\tilde{R}\mathcal F-C_{10}+\frac{\varphi'}2F^{p\bar{p}}_mu_{p\bar{p}}^2,
\end{equation}

for constants $\tilde{R}$ and $C_{10}$ depending on $K, L$ and the geometry of $(X,\omega)$.

\vspace{2ex}

Then, putting together (\ref{eqnIplusIIplusIIIplusV}), (\ref{eqn:IVfinal}) and (\ref{eqn:S_pbarp}), we end up with:

\begin{equation}\label{eqn:Sfinal}
 0\geq -\bar{R}\mathcal F-C_{11}+\psi'F^{p\bar{p}}_m\big(u_{p\bar{p}}+v_{p\bar{p}}\big)+\frac{\varphi'}2F^{p\bar{p}}_mu_{p\bar{p}}^2
\end{equation}

$$=-\bar{R}\mathcal F-C_{11}+\psi'G+|\psi'|F^{p\bar{p}}_mv_{p\bar{p}}+\frac{\varphi'}2F^{p\bar{p}}_mu_{p\bar{p}}^2,$$

where $\bar{R}$ and $C_{11}$ are constants depending on $(X,\omega), K$ and $L$.

To finish the argument, we apply the same idea as in the gradient bound: as $u+v\in[0,L+NC_1r^2]$, we have $|\psi'|\geq\frac1{12(L+NC_1r^2)}$ and choosing $r$ sufficiently small so that

$$24C_1^2(\bar{R}+1)r^2<1$$

allows for a choice of $N$ so that (\ref{eqn:Sfinal}) simplifies to

$$0\geq \mathcal F-C_{12}+\frac{\varphi'}2F^{p\bar{p}}_mu_{p\bar{p}}^2,$$

for $C_{12}$ depending on $L,\ K$ and the geometry of $(X,\omega)$.

The upper bound on $\mathcal F$, coupled with Lemma \ref{lem:ZZ}, leads to uniform upper and lower bounds on each $F^{p\bar{p}}_m$. This, in turn, suffices to bound $u_{1\bar{1}}$ (the largest eigenvalue) above.

\vspace{1ex}

{\bf Step 4: tangential-tangential and tangential-normal estimate on the boundary}.

\vspace{1ex}

As $u$ and $C(||z||^2-r^2)+f$ agree on the boundary, any tangential-tangential second order derivative is trivially bounded by the (first-order) normal one and the geometry of $\partial U$.

The argument for the bound on the mixed tangential-normal derivatives at the boundary is from [Do23, Section 3], where the barrier constructed by Collins and Picard in [CP22] was used. We state the result for the sake of completeness.

\begin{Prop}

Let $\chi$ be an $(n-1)$-semi-positive form on a neighbourhood of $U$ in $(X,\omega)$ and let $u$ solve the problem:

$$F_m[\chi+i\partial\bar{\partial} u](z)=g(z),\hspace{5ex} u|_{\partial U}=f|_{\partial U}.$$

Assume that there is a $C^2$ smooth $m$-subharmonic subsolution $\underline{u}$ solving the problem:

$$F_m[\chi+i\partial\bar{\partial} \underline{u}](z)\geq g(z),\hspace{5ex} \underline{u}|_{\partial U}=f|_{\partial U}.$$

Then, the tangential-normal second-order derivatives of $u$ can be bounded by a constant depending on $(X,\omega), \chi, ||u||_{C^1}, ||\underline{u}||_{C^2}, ||g||_{C^2}$ and $inf_{U} g$.

\end{Prop}

Note that, in our case, we can take $\chi=0$ and $\underline{u}:=f+C(||z||^2-r^2)$ with some $C>0$ depending on $f$. We remark that the result in [Do23] is stated for a solution-independent right-hand side, but once we have $||u||_{C^1}$ under control, $G(z,u(z))$ becomes uniformly positive and bounded, thus the same argument applies.

\vspace{1ex}

{\bf Step 5: normal-normal estimate}.

\vspace{1ex}

Normal - normal estimates for Hessian equations were originally developed in [CNS85].  We follow the exposition from [Do23], which itself relies heavily on [CP22] and the original ideas from [CNS85].

Pick a point $p\in\partial U$ that we identify with the coordinate origin. As in [Do23], we observe that it is no loss of generality to assume that:

\begin{enumerate}

 \item $g_{j\bar{k}}(p)=\delta_{jk}$;

 \item the complex tangent space of $\partial U$ at $p$ is spanned by $\frac{\partial}{\partial z_j},\ j=1,\cdots,n-1$;

 \item $\frac{\partial}{\partial x_n}$ is the inner normal direction, where we use the standard notation $z_j:=x_j+iy_j$;

 \item The local defining function $\rho$ of $\partial U$ near $p$ satisfies $\rho(z)=-x_n+O(|z|^2)$;

 \item the matrix $(u_{j\bar{k}})$ has zero entries except for the diagonal and the last row and column, i.e.

 $$(u_{j\bar{k}})=\begin{pmatrix}

                   u_{1\bar{1}}&0&\cdots&0&u_{1\bar{n}}\\

                   0&u_{2\bar{2}}&\cdots&0&u_{2\bar{n}}\\

                   \vdots&\vdots&\ddots&\vdots&\vdots\\

                   0&0&\cdots&u_{n-1\overline{n-1}}&u_{n-1\bar{n}}\\

                   u_{n\bar{1}}&u_{n\bar{2}}&\cdots&u_{n\overline{n-1}}&u_{n\bar{n}}

                  \end{pmatrix}.
$$

\end{enumerate}

Recall that $F_m[i\partial\bar{\partial}u]:=(\Pi_{|I|=m}(\lambda_{i_1}+\cdots+\lambda_{i_m}))^{1/\binom{n}{m}}$, where $\lambda_j$'s are the eigenvalues of the matrix $(g^{s\bar{l}}u_{q\bar{l}})_{s,q=1}^n$, which at $p$ simplifies to the matrix $(u_{s\bar{q}})$.

In order to exploit this, we recall the following fact, essentially contained in [HL13]: for any Hermitian $n\times n$ matrix $A$, the derivation

$D_A:\Lambda^m\mathbb C^n\longmapsto\Lambda^m\mathbb C^n $ given by the formula

\begin{equation}\label{eqn:D_A}
 D_A(v_1\wedge v_1\cdots\wedge v_m):=(Av_1)\wedge v_2\wedge\cdots\wedge v_m+v_1\wedge(Av_2)\wedge\cdots\wedge v_m+\cdots+ v_1\wedge v_2\wedge\cdots\wedge(Av_m)
\end{equation}

(and extended linearly to indecomposable elements) is a linear map with eigenvalues being precisely the sums of $m$-tuples of eigenvalues of $A$. This immediately leads to the formula

\begin{equation}\label{eqn:FmfromD_A}
F_m[i\partial\bar{\partial}u]=(det(D_{g^{-1}(u_{j\bar{k}})}))^{1/\binom{n}{m}}.
\end{equation}

(note that $g^{-1}(u_{j\bar{k}})$ is Hermitian at $p$.)

Let us specialise to the case of the point $p$ and the basis of $\Lambda^m\mathbb C^n$ given by $\frac{\partial}{\partial z_{i_1}}\wedge\frac{\partial}{\partial z_{i_2}}\wedge\cdots\wedge\frac{\partial}{\partial z_{i_m}}$ for all $m$-tuples $1\leq i_1<i_2<\cdots<i_m\leq n$. Notice that, due to the special shape of $u_{j\bar{k}}(p)$, the entry of the $\binom{n}{m}\times\binom{n}{m}$ matrix $D_{g^{-1}(p)(u_{j\bar{k}}(p))}$ in the slot associated with the pair $(i_1,\cdots,i_m),\ (j_1,\cdots,j_m)$ equals:

$$\begin{cases}

   u_{i_1\overline{i_1}}+\cdots +u_{i_m\overline{i_m}}\ \ {\rm on\ the\ diagonal};\\

   0\ \ {\rm if}\ i_m\neq n\ {\rm or}\ j_m\neq n,

  \end{cases}
$$

and a linear combination of tangential-normal derivatives of $u$ at $p$ (which are already uniformly bounded by previous considerations).

Next, we rewrite (\ref{eqn:FmfromD_A}) in the following form:

\begin{equation}\label{eqn:finalFm}
 G=F_m[i\partial\bar{\partial}u](p)=(det \tilde{D}_{(u_{j\bar{k}}(p))})^{1/\binom{n}{m}}\Pi_{|I|=m,\ i_m=n}(u_{i_1\overline{i_1}}+\cdots +u_{i_m\overline{i_m}})^{1/\binom{n}{m}},
\end{equation}

where $\tilde{D}_{(u_{j\bar{k}}(p))}$ is the matrix  obtained from ${D}_{(u_{j\bar{k}}(p))}$ by dividing each row $(i_1,\cdots, i_m)$ with $i_m=n$ by $\sqrt{u_{i_1\overline{i_1}}+\cdots +u_{i_m\overline{i_m}}}$ and each column $(j_1,\cdots,j_m)$ with $j_m=n$ by $\sqrt{u_{j_1\overline{j_1}}+\cdots +u_{j_m\overline{j_m}}}$.

The core of the proof, just as in [Do23] is to prove the following claim:

\begin{Claim}\label{Claim:prod_uj_c0} There is a constant $c_0$ depending on $(X,\omega),\underline{u}$ and $G$ such that

\begin{equation*}\Pi_{|I|=m,\ i_m<n}(u_{i_1\overline{i_1}}+\cdots +u_{i_m\overline{i_m}})\geq c_0.
\end{equation*}

\end{Claim}

Indeed, assume the claim. Assume additionally that $u_{n\bar{n}}=:M$ is very large. Then, (\ref{eqn:finalFm}) can be rewritten as

$$G^{\binom{n}{m}}=(O(1)+M)^{\binom{n-1}{m-1}}det(\tilde{D}_{(u_{j\bar{k}}(p))})$$

with the off-diagonal entries of $\tilde{D}_{(u_{j\bar{k}}(p))}$ being either zero or of order $O(\frac1{\sqrt{M}})$ i.e. the matrix is {\it almost} diagonal. Additionally, diagonal entries are either $1$ or  of the form $u_{i_1\overline{i_1}}+\cdots +u_{i_m\overline{i_m}}$ with $i_m<n$. Thus:

$$det(\tilde{D}_{(u_{j\bar{k}}(p))})=\Pi_{|I|=m,\ i_m<n}(u_{i_1\overline{i_1}}+\cdots +u_{i_m\overline{i_m}})-O\bigg(\frac1{\sqrt{M}}\bigg)\geq\frac{c_0}2$$

for $M$ large enough. As $G$ is bounded above, such an equality cannot hold for arbitrarily large $M$.

Thus, one obtains an upper bound for $u_{n\bar{n}}(p)$ and hence on $u_{x_nx_n}(p)=u_{n\bar{n}}-u_{y_ny_n}$ as $\frac{\partial}{\partial y_n}$ is a real tangential direction. But $u$ is additionally subharmonic (with respect to $g$) and thus:

$$u_{x_nx_n}(p)\geq -\sum_{j=1}^{n-1}u_{x_jx_j}(p)-\sum_{k=1}^nu_{y_ky_k}(p)\geq -C.$$

This provides a lower bound for the normal-normal derivative. Hence, the estimate for the normal-normal derivative is reduced to the proof of Claim \ref{Claim:prod_uj_c0}.

\vspace{2ex}

\noindent {\it Proof of Claim \ref{Claim:prod_uj_c0}.} From now on the proof repeats the arguments from [Do23]. We provide the details for the sake of completeness.

 For every base point $z$ and every $k\geq m$, let us denote by $P_m^k(z)$ the set of $(k\times k)$-dimensional $g$-Hermitian matrices for which any sum of $m$ eigenvalues (with respect to $g(z)$) is non-negative. Equivalently (compare [HL13, Definition 2.1]), $A\in P_m^k$ iff the trace $tr_W(A)$ of $A$ restricted to any $m$-dimensional complex subspace $W$ is non-negative. Let us denote by $P_m^n(B)(z)$ the set of matrices in $P_m^n$ such that

\begin{equation}\label{eqn:PmnB}
 P_m^n(B)(z):=\lbrace A\in P_m^n(z)\ |\ F_m[A]\geq F_m[B]\rbrace.
\end{equation}

Recall that $\underline{u}$ is an explicit subsolution, so $\eta_0:=(u-\underline{u})_{x_n}(p)>0$ and $\eta_0$ is uniformly bounded.

For $j,k\in\lbrace 1,\cdots,n-1\rbrace$, we have:

\begin{equation}\label{eqn:uundelineu}
 (u-\underline{u})_{j\bar{k}}(p)=-(u-\underline{u})_{x_n}(p)\rho_{j\bar{k}}(p).
\end{equation}

In particular, for $t=1$, the matrix:

$$A_t:=(t\underline{u}_{j\bar{k}}(p)-(u-\underline{u})_{x_n}(p)\rho_{j\bar{k}}(p))_{j,k=1}^{n-1}$$

belongs to $P_m^{n-1}$ (as it is equal to $({u}_{j\bar{k}}(p))_{j,k=1}^{n-1}$). Denote by $t_0$ the smallest $t$ for which such a property holds. Note that

$$({u}_{j\bar{k}}(p))_{j,k=1}^{n-1}=(1-t_0)(\underline{u}_{j\bar{k}}(p))_{j,k=1}^{n-1}+A_{t_0}.$$

For any $m$-tuple $I=(i_1,\dots , i_m)$ with $i_j\in\lbrace1,\cdots,n-1\rbrace$, consider the subspace $W_I$ spanned by $\frac{\partial}{\partial z_{i_1}},\cdots,\frac{\partial}{\partial z_{i_m}}$. We have:

$$u_{i_1\overline{i_i}}+\cdots+u_{i_m\overline{i_m}}=tr_{W_I}((u_{j\bar{k}})_{j,k=1}^{n-1})=(1-t_0)tr_{W_I}((\underline{u}_{j\bar{k}})_{j,k=1}^{n-1})+tr_{W_I}A_{t_0}\geq (1-t_0)c_1>0$$

for a constant $c_1$ which depends only on $\underline{u}$ and is, hence, uniform. Therefore, the claim follows from a uniform lower bound for $(1-t_0)$.

Just as in [Do23], we consider the auxiliary functions:

$$D(z)=-\rho(z)+\tau|z|^2,$$

$$\Phi(z)=\underline{u}(z)-\frac{\eta_0}{\tau_0}\rho(z)+2\Re[\sum_{j=1}^{n-1}l_jz_j]\rho(z)+LD(z)^2,$$

$$\Psi(z)=\Phi(z)+\varepsilon(|z|^2-\frac{x_n}{C}),$$

all of which are defined in a small coordinate ball $B_{\delta}(0)$ centred at $p$ (identified with the coordinate origin) and of radius $\delta>0$. Here, $\tau>0$ and $\varepsilon>0$ are small constants, $L, C>1$ are large ones and $l_j\in\mathbb C$ are bounded numbers. All these quantities will chosen later on.

\vspace{2ex}

We now state a result analogous to Lemma 4.2 from [Do23]:

\begin{Lem}\label{Lem:Psimajorizesu}

For a specific choice of the aforementioned constants, one has:

$$\forall z\in U\cap B_{\delta}(0)\ \ u(z)\leq\Psi(z).$$

\end{Lem}

Assuming this lemma, the last part of the proof of Claim \ref{Claim:prod_uj_c0} proceeds as follows. Since

$$\Psi(0)=\Phi(0)=\underline{u}(0)+LD(0)^2=\underline{u}(0)=u(0),$$ Lemma \ref{Lem:Psimajorizesu} implies that $(\Psi-u)_{x_n}(0)\geq 0$ which, in turn, can be rewritten as

\begin{equation}\label{eqn:rewritenormalPsiminusu}
-\eta_0=-(u-\underline{u})_{x_n}(0)\geq\frac{\varepsilon}{C}+\frac{\eta_0}{t_0}\rho_{x_n}(0)=\frac{\varepsilon}{C}-\frac{\eta_0}{t_0},
\end{equation}

or, more concisely, as

$\frac{1-t_0}{t_0}\geq\frac{\varepsilon}{C\eta_0}$. The lower bound on $(1-t_0)$ follows immediately from this.

This completes the proof of Claim \ref{Claim:prod_uj_c0}.  \hfill $\Box$

\vspace{2ex}

Thus, we have come to the last part of the argument.

\vspace{2ex}

\noindent {\it Proof of Lemma \ref{Lem:Psimajorizesu}.} We refer the reader to [Do23, Lemma 4.2] for a more detailed exposition.

On $\partial B_{\delta}(0)\cap U$, we have: $$\Psi-u=\big[\underline{u}-u\big]+\big[\varepsilon(|z|^2-\frac{x_n}{C})-\frac{\eta_0}{\tau_0}\rho(z)+2\Re[\sum_{j=1}^{n-1}l_jz_j]\rho(z)\big]+LD(z)=:\alpha+\beta+\gamma.$$

Now, $\alpha$ is bounded below by $-C_{13}$, which depends on the uniform bounds of $u$ and $\underline{u}$. Meanwhile, $\beta$ is bounded by $-C_{14}$, which depends only on the $l_j$'s if we take $C$ larger than $\frac1\delta$. Finally, $\gamma\geq L\tau^2\delta^4$. Hence, the choice $L:=\frac{C_{13}+C_{14}}{\tau^2\delta^4}$ (or any larger constant) leads to $\Psi-u\geq 0$ over this part of the boundary.

On $B_\delta(0)\cap \partial U$, we have:

$\alpha=0$ and $\gamma=L\tau^2|z|^4$. Exploiting that $\rho(z)=-x_n+O(|z|^2)$, we get: $$\beta=\varepsilon\bigg(|z|^2-\frac{x_n}{C}\bigg)=\varepsilon\bigg(|z|^2-\frac{O(|z^2|)}{C}\bigg),$$ which is non-negative once $C$ has been chosen large enough. Hence:

$$\Psi-u\geq 0\ {\rm on}\ \partial(B_{\delta}(0)\cap U).$$

The inequality would extend to the interior once a  version of the maximum principle can be applied. To this end, we shall exploit Lemma 2.1 from [Do23]: it suffices to show that $(g^{-1}(\Psi_{j\bar{k}})_{j,k=1}^n)(z)$ does not belong to $P_m^n[(u_{j\bar{k}})_{j,k=1}^n(z)]$ for any $z$ lying in $B_{\delta}(0)\cap U$.

Observe that $$(g^{-1}(\Psi_{j\bar{k}})_{j,k=1}^n)(z)=(g^{-1}(\Phi_{j\bar{k}}+\varepsilon Id)_{j,k=1}^n)(z).$$

Since $F_m[i\partial\bar{\partial}u]=G$ is uniformly bounded below, the relation $(g^{-1}(\Psi_{j\bar{k}})_{j,k=1}^n)(z)\notin P_m^n[(u_{j\bar{k}})_{j,k=1}^n(z)]$ follows for sufficiently small $\varepsilon>0$ from the continuity of the eigenvalues with respect to the matrix entries once we know that $(g^{-1}(\Phi_{j\bar{k}})_{j,k=1}^n(z))$ does not belong to the interior of $P_m^n$.

To prove this, we assume that $\frac12\leq t_0\leq 1$ for any $z$ close to the origin and we construct an $m$-dimensional complex subspace $W_z$ on which $tr_{W_{z}}(g^{-1}(\Phi_{j\bar{k}})_{j,k=1}^n(z))\leq 0$.

As $A_{t_0}$ is at the boundary of $P_m^{n-1}$, there is such a complex subspace $W_0$ of $T_0\partial U$ such that

\begin{equation}\label{eqn:trAt0}
 0=tr_{W_0}A_{t_0}=t_0tr_{W_0}D^2_{\mathbb C}\Phi(0)=tr_{W_0}D^2_{\mathbb C}[t_0\underline{u}-\eta_0\rho].
\end{equation}

 (Note that the third and fourth terms in the definition of $\Phi$ have vanishing Levi forms at $T_0\partial U$). Let $\xi_1(0),\cdots,\xi_m(0)$ be an orthonormal frame with respect to $g(0)=Id$ that spans $W_0$ . Just as in [Do23], we extend these to local $g$-orthonormal vector fields $\xi_j(z)=\sum_{k=1}^n\xi_j^k(z)\frac{\partial}{\partial z_k}$ in $B_{\delta}(0)\cap U$ such that the $\xi_j(z)$ (with $j=1,\cdots,m$) span an $m$-dimensional complex subspace $W_z$ of the tangent space to the level set of $D(z)$. Observe the equality:

$$tr_{W_z}g^{-1}(z)D^2_{\mathbb C} \Phi(z):=\sum_{j=1}^m\sum_{p,q=1}^n\xi_j^p\overline{\xi_j^q}\Phi_{p\bar{q}}$$

$$=\sum_{j=1}^m\sum_{p,q=1}^n\xi_j^p\overline{\xi_j^q}(\big[\underline{u}(z)-\frac{\eta_0}{\tau_0}\rho(z)\big]_{p\bar{q}}+\big[2\Re[\sum_{s=1}^{n-1}l_sz_s]\rho(z)\big]_{p\bar{q}}+\big[LD(z)^2\big]_{p\bar{q}})$$

$$=:T_1+T_2+T_3.$$

As $T_1(0)=0$ and $T_1$ is a smooth real function depending on $\underline{u}$ and $\rho$, we have:

\begin{equation}\label{eqn:T1}
 T_1(z)=2\Re(\sum_{j=1}^nm_jz_j)+O(|z|^2)
\end{equation}

for some uniform constants $m_j\in\mathbb C$ under control.

To estimate $T_2$, we exploit the fact that the $\xi_j$'s are tangent to the level sets of $D(z)$ and hence

\begin{equation}\label{eqn:xiD}
 0=\xi_jD=-\xi_j^k\rho_k+\tau\xi_j^k\bar{z}_k.
\end{equation}

Taking (\ref{eqn:xiD}) into account, the computation of $T_2$ and the Taylor expansion at $0$ reveal that \begin{equation}\label{eqn:T2}
 T_2=2\Re[\sum_{s=1}^{n-1}l_sz_s]tr_{W_0}D^2_{\mathbb C}\rho(0)+\tau\sum_{j=1}^m\sum_{p,q=1}^n\xi_j^p(0)\overline{\xi_j^q}(0)(l_pz_q+\bar{l}_q\bar{z}_p)+O(|z|^2).
\end{equation}

Coupling now (\ref{eqn:T1}) with (\ref{eqn:T2}), we get:

$T_1+T_2=O(|z|^2)$ provided that the numbers $l_j\in\mathbb C$ satisfy the linear system of equations:

$$m_j+tr_{W_0}D^2_{\mathbb C}\rho(0)l_j+\tau\sum_{s=1}^{m}\sum_{p=1}^n\xi_s^p(0)\overline{\xi_s^j}(0)l_p=0.$$ Now, (\ref{eqn:trAt0}) implies that $tr_{W_0}D^2_{\mathbb C}\rho(0)=\frac{t_0}{\eta_0}tr_{W_0}D^2_{\mathbb C}\underline{u}(0)$, which is uniformly positive (i.e. bounded below by some constant $\theta>0$). Hence, this system is solvable for all sufficiently small $\tau$.

Fixing such a $\tau>0$ which additionally satisfies:

\begin{equation}\label{eqn:tau}
\tau\leq \frac{\theta}{2m}
\end{equation}

and the corresponding $l_j$'s, we finally observe that: $$T_3=2LD(z)\sum_{j=1}^m\sum_{p,q=1}^n\xi_j^p(z)\overline{\xi_j^q}(z)D_{p\bar{q}}(z).$$

(We have used, once again, that the $\xi_j$'s are tangential to the level sets of $D$.)

Moreover, at the origin, $\sum_{p,q=1}^n\xi_j^p(z)\overline{\xi_j^q}(z)D_{p\bar{q}}(z)$ equals

$-tr_{W_0}D^2_{\mathbb C}\rho+\tau\sum_{s=1}^m||\xi_s(0)||^2\leq -\frac{\theta}2$, where the last inequality follows from (\ref{eqn:tau}) and the orthonormality of the $\xi_j$'s. Hence, $$\sum_{p,q=1}^n\xi_j^p(z)\overline{\xi_j^q}(z)D_{p\bar{q}}(z)\leq-\frac{\theta}{4}$$ near the origin and thus

\begin{equation}\label{eqn:T4}
 T_3\leq \frac{-\tau\theta L}{2}|z|^2.
\end{equation}

Finally,

$$ tr_{W_z}D^2_{\mathbb C} \Phi(z)=T_1+T_2+T_3=O(|z|^2)-\frac{-\tau\theta L}{2}|z|^2\leq 0,$$

provided $L$ is taken large enough.

This conlcudes the proof of Lemma \ref{Lem:Psimajorizesu}, hence the proof of Theorem \ref{Thm:Dirichlet-Fm}.

\end{proof}

\vspace{3ex}

\noindent {\bf References.} \\

\noindent [Ber19]\, R. Berman --- {\it From Monge-Amp\`ere equations to envelopes and geodesic rays in the zero temperature limit} --- Math. Zeit. {\bf 291}, 1-2 (2019), 365-394.

\vspace{1ex}

\noindent [BK07]\, Z. B\l{}ocki, S. Ko\l{}odziej --- {\it On Regularization of Plurisubharmonic Functions on Manifolds} --- Proc. Amer. Math. Soc. {\bf 135} (2007), 2089–2093.

\vspace{1ex}

\noindent [BP25]\, K. Broder, D. Popovici --- {\it  Holomorphic $p$-Contact and $s$-Symplectic Line Bundles}--- preprint arXiv 2511.23139v2, to appear in the Journal of Geometric Analysis.



\vspace{1ex}

\noindent [CNS85]\, L. Caffarelli, L. Nirenberg, J. Spruck, --- {\it The Dirichlet problem  for nonlinear second order elliptic equations, III: Functions of the
eigenvalues of the Hessian} --- Acta Math. {\bf 155} (1985), no. 3-4, 261-301.

\vspace{1ex}

\noindent [ChX25]\, J. Cheng, Y. Xu --- {\it Viscosity solution to complex Hessian equations on compact Hermitian manifolds} --- J. Funct. Anal. {\bf 289}, 5 (2025), art. 110936.

\vspace{1ex}

\noindent [CP22]\, {T. C. Collins, S. Picard, --- {\it The Dirichlet problem for the $k$-Hessian equation on a complex manifold} --- Amer. J. Math. {\bf 144} (2022), 1641-1680.}

\vspace{1ex}

\noindent [Dem 84]\, J.-P. Demailly --- {\it Sur l'identit\'e de Bochner-Kodaira-Nakano en g\'eom\'etrie hermitienne} --- S\'eminaire d'analyse P. Lelong, P. Dolbeault, H. Skoda (editors) 1983/1984, Lecture Notes in Math., no. {\bf 1198}, Springer Verlag (1986), 88-97.

\vspace{1ex}

\noindent [Dem92]\, J.-P. Demailly --- {\it Regularization of Closed Positive Currents and Intersection Theory} --- J. Alg. Geom., {\bf 1} (1992), 361-409.

\vspace{1ex}

\noindent [Dem97]\, J.-P. Demailly --- {\it Complex Analytic and Algebraic Geometry} --- \url{https://www-fourier.univ-grenoble-alpes.fr/~demailly/manuscripts/agbook.pdf}



\vspace{1ex}

\noindent [Die06]\, N. Q. Dieu --- {\it $q$-Plurisubharmonicity and $q$-Pseudoconvexity in $\C^n$} --- Publ. Mat., Barc. {\bf 50} (2006), no. 2, p. 349-369.

\vspace{1ex}

\noindent [Din22]\, S. Dinew --- {\it $m$-subharmonic and $m$-plurisubharmpnic functions: on two problems of Sadullaev} --- Ann. Fac. Sci. Toulouse Math. (6) {\bf 31} (2022), no. 3, 995-1009.

\vspace{1ex}

\noindent [Din23]\, S. Dinew --- {\it Interior estimates for $p$-plurisubharmonic functions} --- Indiana Univ. Math. J.  {\bf 72} (2023), no. 5, 2025-2057.

\vspace{1ex}

\noindent [DP25]\, S. Dinew, D. Popovici --- {\it $m$-Pseudo-effectivity and a Monge-Amp\`ere-Type Equation for Forms of Positive Degree} --- arXiv e-print DG 2510.27362v1.

\vspace{1ex}

\noindent[Do23]\, W. Dong --- {\it The Dirichlet problem for Monge-Amp\`ere type equations on Hermitian manifolds} --- Discr. Cont. Dynam. Syst. {\bf 43} (2023), no. 11, 3925-3939.

\vspace{1ex}

\noindent[GGQ22]\, M. George, B. Guan, C. Qiu --- {\it Fully nonlinear Elliptic Equations on Hermitian manifolds for Symmetric Functions of Partial Laplacians} --- J. Geom. Anal. {\bf 32} (2022) Paper no. 183.

\vspace{1ex}

\noindent[GN18]\, {D. Gu, N. C. Nguyen --- {\it The Dirichlet problem for a complex Hessian equation on compact Hermitian manifolds with boundary}

--- Ann. Sc. Norm. Super. Pisa Cl. Sci. {\bf 18} (2018), no. 4, 1189-1248.}

\vspace{1ex}

\noindent [HL13]\, F.R. Harvey, H.B. Lawson --- {\it p-Convexity, p-Plurisubharmonicity and the Levi Problem} --- Indiana Univ. Math. J. {\bf62} (2013), no. 1, 149-169.



\vspace{1ex}

\noindent [HLP16]\, F.R. Harvey, H.B. Lawson, Sz. Pli\'s --- {\it Smooth approximation of plurisubharmonic functions on almost complex manifolds} --- Math. Ann. {\bf 366}, 3 (2016), 929-940.



\vspace{1ex}

\noindent [LN15]\, H. C. Lu, V. D. Nguyen --- {Degenerate complex Hessian equations on compact K\"ahler manifolds} --- Indiana Univ. Math. J. {\bf 64}, 6 (2015),  1721-1745.

\vspace{1ex}

\noindent [Pli13]\, Sz. Pli\'s --- {\it The smoothing of $m$-subharmonic functions} --- preprint arXiv 1312.1906v2.

\vspace{1ex}

\noindent [Pop25]\, D. Popovici --- {\it Generalised Hermite-Einstein Fibre Metrics and Slope Stability for Holomorphic Vector Bundles} --- arXiv:2512.24932v2 [math.AG].





\vspace{1ex}

\noindent [Ver10]\, M. Verbitsky --- {\it Plurisubharmonic Functions in Calibrated Geometry and q-Convexity} --- Math. Z. {\bf 264} (2010), no. 4, p. 939-957.



\vspace{1ex}

\noindent [Yan18]\, X. Yang --- {\it A Partial Converse to the Andreotti-Grauert Theorem} --- Compos. Math. {\bf 155} (2019), no. 1, 89–99.

\vspace{1ex}

\noindent [Yau78]\, S.T. Yau --- {\it On the Ricci Curvature of a Complex K\"ahler Manifold and the Complex Monge-Amp\`ere Equation I} --- Comm. Pure Appl. Math. {\bf 31} (1978) 339-411.

\vspace{6ex}

\noindent Department of Mathematics and Computer Science     \hfill Institut de Math\'ematiques de Toulouse,

\noindent Jagiellonian University     \hfill  Universit\'e de Toulouse,

\noindent 30-409 Krak\'ow, Ul. Lojasiewicza 6, Poland    \hfill  118 route de Narbonne, 31062 Toulouse, France

\noindent  Email: Slawomir.Dinew@im.uj.edu.pl                \hfill     Email: popovici@math.univ-toulouse.fr

\end{document}